\documentclass[11pt]{amsart}

\usepackage{geometry}                % See geometry.pdf to learn the layout options. There are lots.
\usepackage{graphicx}  
\usepackage{subfig}
\usepackage{labelfig}
\usepackage{amssymb}
\usepackage{color}
\newtheorem{theorem}{Theorem}[section]    % Standard theorem environment
\newtheorem{lemma}[theorem]{Lemma}          % Lemma environment with numbering 
\newtheorem{proposition}[theorem]{Proposition}

\newtheorem{corollary}[theorem]{Corollary} 
\theoremstyle{definition}
\newtheorem{definition}[theorem]{Definition}

\newtheorem{remark-number}[theorem]{Remark}  
\newtheorem{question}[theorem]{Question} 
\newtheorem{observation}[theorem]{Observation} 
\newtheorem{example}[theorem]{Example}   % Definition environment with 
\numberwithin{equation}{section}
\newcommand{\NL}{NL}
\newcommand{\F}{\mathcal F_{ob} }

\newcommand{\Int}{\textrm{Int}}
\newcommand{\mC}{\mathcal{C}}
\newcommand{\mF}{\mathcal{F}}

\newcommand{\Z}{\mathbb{Z}}
\newcommand{\sgn}{{\tt sgn}}
\newcommand{\e}{\varepsilon}

\begin{document}

\title{Overtwisted discs in planar open books}
\author{Tetsuya Ito}
\address{Research Institute for Mathematical Sciences, Kyoto university, Kyoto, 606-8502, Japan}
\email{tetitoh@kurims.kyoto-u.ac.jp}
\urladdr{http://www.kurims.kyoto-u.ac.jp/~tetitoh/}
\author{Keiko Kawamuro}
\address{Department of Mathematics,   
The University of Iowa, Iowa City, IA 52242, USA}
\email{kawamuro@iowa.uiowa.edu}
\date{\today}

\begin{abstract}
Using open book foliations we show that an overtwisted disc in a planar open book can be put in a topologically nice position. As a corollary, we prove that a planar open book whose fractional Dehn twist coefficients grater than one for all the boundary components supports a tight contact structure. 
\end{abstract}

\maketitle
%\tableofcontents

\section{Introduction}

There is a rigid dichotomy between {\em tight} and {\em overtwisted} contact structures on $3$-manifolds. 
All the contact structures are locally identical, so tightness and overtwistedness are global properties. 

Eliashberg's classification of  overtwisted contact structures \cite{el} states that overtwisted contact structures are classified by the homotopy types of $2$-plane fields. On the other hand, tight contact structures are more subtle, and  classification of tight contact structures is still open except for several simple cases including some Seifert fibered spaces.

It is often hard to determine whether a given contact structure is tight or overtwisted. 
Some types of fillability, such as (weak and strong) symplectic or Stein, imply tightness of the contact structures. Non-vanishing of Ozsv\'ath and Szab\'o's contact invariant shows tightness \cite{OS}. 
In convex surface theory, Giroux's criterion \cite{g0} is useful to find overtwisted discs, and Honda's state traversal method \cite{ho1} provides a way to prove tightness. 
In \cite{et} Eliashberg and Thurston use confoliations to prove that a contact structure obtained by $C^{0}$-small perturbation of a taut foliation is universally tight (cf. \cite{hkm2}).  

In this paper we give a new tightness criterion, Corollary~\ref{cor:tight}, using strong topological and combinatorial aspects of open book foliations \cite{ik1-1,ik1-2,ik2,ik3}. 
Here is our main theorem: 

\begin{theorem}
\label{theorem:main}
Let $(S,\phi)$ be a planar open book which supports an overtwisted contact structure. Then there exists a transverse overtwisted disc $D$ such that: 
\begin{description}
\item [(SE1)] All the valence $\leq 1$ vertices of the graph $G_{--}(D)$ are strongly essential. 
\end{description}
\end{theorem}

The graph $G_{--}(D)$ and strongly essential vertices are defined in Section \ref{sec:review}. Roughly speaking, Theorem~\ref{theorem:main} shows that we can put a transverse overtwisted disc so that it intersects each page of the open book  in some nice way. 
In the theory of Haken $3$-manifolds, one uses essential surfaces to cut $3$-manifolds and study the structure and properties of the manifolds. 
We apply this classical scheme to contact $3$-manifolds and surfaces admitting open book foliations and analyze topological features of the open books.

For a boundary component $C \subset \partial S$ let $c(\phi,C)$ denote the {\em fractional Dehn twist coefficient} (FDTC) of $\phi$ with respect to $C$. See \cite{hkm1} for the definition. 
In \cite[Theorem 1.1]{hkm1} Honda, Kazez and Mati\'c prove that an open book $(S, \phi)$ supporting a tight contact structure implies that $\phi$ is right-veering, in particular, $c(\phi,C)\geq 0$ for all the boundary components $C$ of $S$. The next corollary asserts the converse direction of the theorem under some assumptions on FDTC.

\begin{corollary}
\label{cor:tight}
Let $(S,\phi)$ be a planar open book. If $c(\phi,C)>1$ for all the boundary components $C \subset \partial S$ then $(S,\phi)$ supports a tight contact structure.
\end{corollary}

It is interesting to compare Corollary \ref{cor:tight} with the result of Colin and Honda in \cite{ch}. They show that for a (not necessarily planar) open book $(S,\phi)$ with pseudo-Anosov monodromy, if $c(\phi,C_i)\geq \frac{k}{n_i}$ ($k\geq 2$) for every boundary component $C_i$ of $S$, where $n_i$ is the number of prongs around $C_i$ of the transverse measured (stable) foliation for $\phi$, then $(S,\phi)$ supports a universally tight contact structure (\cite{ch} treats the connected binding case, and by Baldwin and Etnyre \cite[Theorem 4.5]{be} the same result holds for the general case). 
They show tightness by proving non-vanishing of the contact homology. 
Note that the foundation of contact homology requires hard analysis and is geometric in the sense that its definition uses Reeb vector fields and contact forms. 
On the other hand, our argument using open book foliations is topological by nature. We do not need to determine the Nielsen-Thurston types of monodromies. We just add a topological assumption that the page surface is planar.

\begin{remark-number}   
\label{rem:bp}
Let $S=S_{0, 4}$ be a sphere with four discs removed. Call the boundary circles $A,B,C$ and $D$. Let $E$ be a simple closed curve in $S$ that separates $A, B$ from $C, D$. For $h,i,k>0$, let $\Phi_{h,i,k}=T_{A}^{h}T_{B}^{i}T_{C}T_{D}T_{E}^{-k-1}$, where $T_{X}$ denotes the right-handed Dehn twist along $X \in \{A,B,C,D,E\}$.
In \cite[Theorem 4.1]{ik1-2} we show that the open book $(S,\Phi_{h,i,k})$ is non-destabilizable and supports an overtwisted contact structure. 
The FDTCs of $\Phi=\Phi_{h,i,k}$ are 
\[ (c(\Phi,A),c(\Phi,B),c(\Phi,C),c(\Phi,D))=(h,i,1,1). \]
Moreover,
James Conway, John Etnyre, Amey Kaloti, and Dheeraj Kulkarni found a non-destabilizable open book $(S,\Psi)$ with $\Psi=T_{A}T_{B}^{2}T_{C}^{2}T_{D}^{3}T_{E}^{-2}$ supporting an overtwisted contact structure.
The FDTCs of $\Psi$ are
\[(c(\Psi,A),c(\Psi,B),c(\Psi,C),c(\Psi,D))=(1,2,2,3). \]
Thus the conditions in Corollary \ref{cor:tight} are best possible even if we add a reasonable assumption that $(S,\phi)$ is non-destabilizable.
\end{remark-number}

\section{Review of open book foliations} 
\label{sec:review}

In this section we summarize definitions and properties of open book foliations used in this paper.  For details, see \cite{ik1-1,ik2,ik3}. The idea of open book foliations originally came from Bennequin's work \cite{Ben} and Birman-Manasco's braid foliations  \cite{BM4, BM2, BM5, BM1, BM6, BM3, BM7, bm1, bm2}.

Let $S=S_{g, r}$ be a genus $g$ surface with $r(>0)$ boundary components, and $\phi \in {\rm Diff}^{+}(S, \partial S)$ an orientation preserving differomorphism of $S$ fixing the boundary $\partial S$ point-wise.
Suppose that the open book $(S,\phi)$ supports a closed oriented contact $3$-manifold $(M,\xi)$ via the Giroux correspondence \cite{g}. The manifold $M$ is often denoted by $M_{(S, \phi)}$. 
Let $B$ denote the {\em binding} of the open book and $\pi:M \setminus B \rightarrow S^{1}$ the fibration whose fiber $S_{t} := \pi^{-1}(t)$ is called a {\em page}.

Let $F \subset M_{(S, \phi)}$ be an embedded, oriented surface possibly with boundary. If $F$ has boundary we require that $\partial F$ is a closed braid with respect to $(S, \phi)$, that is, $\partial F$ is positively transverse to every page. Up to isotopy of $F$ fixing $\partial F$ we may put $F$ so that the singular foliation given by the intersection with the pages
\[ \F(F)=\left\{ F \cap S_t \ | \ t \in [0, 1] \right\} \] 
admits the following conditions {\bf ($\mF$ i)}--{\bf ($\mF$ iv)}, see \cite[Theorem 2.5]{ik1-1}. 
We call $\F(F)$ an {\em open book foliation} on $F$. 
\begin{description}
\item[($\mF$ i)] 
The binding $B$ pierces the surface $F$ transversely in finitely many points. 
Moreover, $v \in B \cap F$ if and only if there exists a disc neighborhood $N_{v} \subset \Int(F)$ of $v$ on which the foliation $\F(N_v)$ is radial with the node $v$, see Figure~\ref{fig:sign}-(1, 2). The leaves meeting at $v$ belong to distinct pages. We call $v$ an {\em elliptic} point. 
\item[($\mF$ ii)] 
The leaves of $\F(F)$ along $\partial F$ are transverse to $\partial F$. 
\item[($\mF$ iii)] 
All but finitely many pages $S_{t}$ intersect $F$ transversely.
Each exceptional fiber is tangent to $F$ at a single point $\in\Int(F)$.
In particular, $\F(F)$ has no saddle-saddle connections.
\item[($\mF$ iv)] 
All the tangencies of $F$ and pages are of saddle type, see Figure~\ref{fig:sign}-(3, 4). 
We call them {\em hyperbolic} points.
\end{description}

A leaf $l$ of $\F(F)$, a connected component of $F \cap S_t$, is called {\em regular} if $l$ does not contain a tangency point, and {\em singular} otherwise.
Regular leaves are classified into the following three types:
\begin{enumerate}
\item[a-arc]: An arc whose one of its endpoints lies on $B$ and the other lies on $\partial F$.
\item[b-arc]: An arc whose endpoints both lie on $B$.
\item[c-circle]: A simple closed curve.
\end{enumerate}

We say that an elliptic point $v$ is {\em positive} (resp. {\em negative}) if the binding $B$ is positively (resp. negatively) transverse to $F$ at $v$.
The sign of the hyperbolic point $h$ is {\em positive} (resp. {\em negative}) if the positive normal direction of $F$ at $h$ agrees (resp. disagrees) with the direction of $t$.
We denote the sign of a singular point $x$ by $\sgn(x)$.
See Figure \ref{fig:sign}, where we describe an elliptic point by a hollowed circle with its sign inside, a hyperbolic point by a black dot with the sign indicated nearby, and positive normals  $\vec n_F$ to $F$ by dashed arrows. 
\begin{figure}[htbp]
 \begin{center}
\SetLabels
(0.03*0.95) (1)\\
(0.58*0.95) (2)\\
(0.03*0.48) (3)\\
(0.58*0.48) (4)\\
(0.2*0.92) $B$\\
(0.72*0.92) $B$\\
(0.08*0.62) $\vec n_F$ \\
(0.6*0.62) $\vec n_F$ \\
(0.19*0.23) $\vec n_F$ \\
(0.71*0.23) $\vec n_F$ \\
(0.75*0.8) $F$\\
(0.23*0.8) $F$\\
(0.03*0.2) $t$\\
(0.55*0.2) $t$\\
(0.22*0.62) $t$\\
(0.74*0.62) $t$\\
(0.24*0.33) $F$\\
(0.77*0.33) $F$\\
\endSetLabels
\strut\AffixLabels{\includegraphics*[scale=0.5, width=120mm]{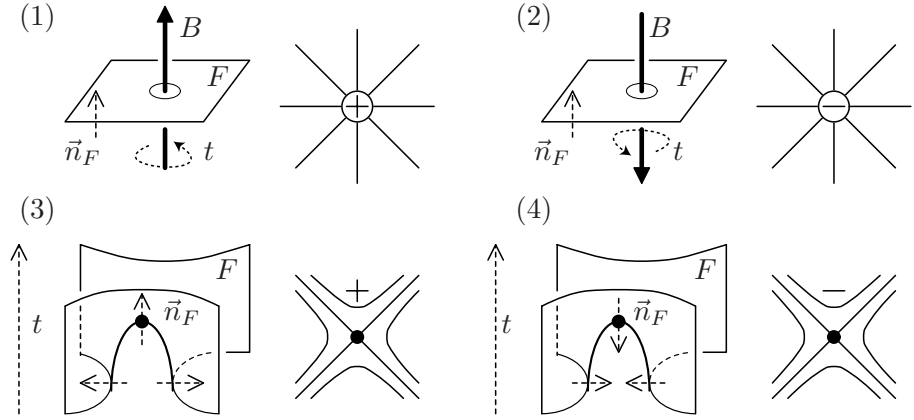}}
 \end{center}
 \caption{Signs of singularities and normal vectors $\vec n_F$.}
  \label{fig:sign}
\end{figure}

The neighborhoods of hyperbolic points as depicted in Figure ~\ref{region} are called {\em regions}. 
There are six types of regions according to the types of nearby regular leaves of the hyperbolic points; 
$aa$-tile, $ab$-tile, $bb$-tile, $ac$-annulus, $bc$-annulus, and $cc$-pants.
\begin{figure}[htbp]
\begin{center}
\SetLabels
(0.15*0.55)  $aa$-tile\\
(0.5*0.55)    $ab$-tile\\
(0.84*0.55)  $bb$-tile\\
(0.15*0.04)  $ac$-annulus\\
(0.5*0.04)    $bc$-annulus\\
(0.84*0.04)  $cc$-pants\\
\endSetLabels
\strut\AffixLabels{\includegraphics*[scale=0.5, width=90mm]{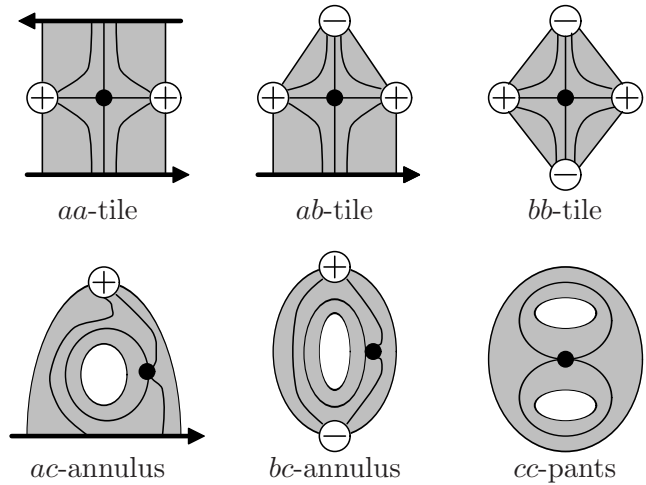}}
\caption{Six types of regions.}\label{region}
\end{center}
\end{figure}

For a region $R$ of type aa, ac, bc, or cc, some components of the boundary $\partial R$ are possibly identified in $F$ (see Figure \ref{degtiles}). In such case we say that $R$ is {\em degenerate}. Certain degenerate regions cannot exist because of {\bf ($\mathcal{F}$ i)}. 
We denote by $\sgn(R)$ the sign of the hyperbolic point contained in the region $R$. The surface $F$ is decomposed into the union of regions whose interiors are disjoint \cite[Proposition 2.15]{ik1-1}.
\begin{figure}[htbp]
\begin{center}
%\ShowGrid
\SetLabels
(0.1*0.54) {\small identified}\\
(0.52*0.54) {\small identified}\\
(0.88*0.54) {\small identified}\\
(0.18*0.02) {\small Degenerate aa-tile}\\
(0.53*0.02) {\small Degenerate bc-annulus}\\
(0.88*0.03) {\small forbidden}\\
(0.03*0.95) {(i)}\\
(0.36*0.95) {(ii)}\\
(0.68*0.95) {(iii)}\\
\endSetLabels
\strut\AffixLabels{\includegraphics*[width=100mm]{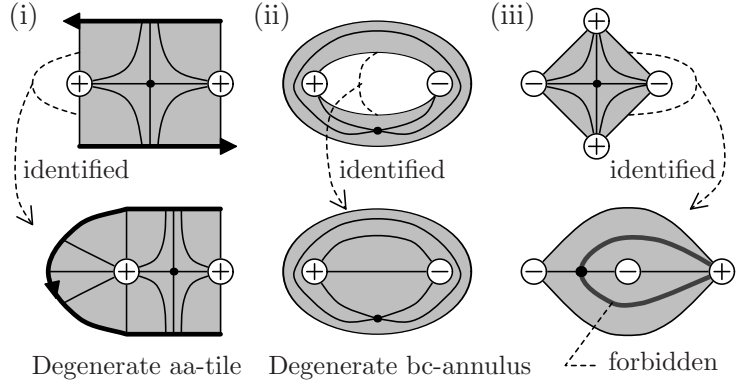}}
\caption{(i, ii) Degenerate regions. (iii) A forbidden region.}\label{degtiles}
\end{center}
\end{figure}

We will often take the following homotopical properties of leaves into account.

\begin{definition}
Let $b$ be a b-arc in $S_{t}$. We say that: 
\begin{enumerate}
\item
$b$ is {\em essential} if $b$ is not boundary-parallel in $S_{t}\setminus(S_{t} \cap \partial F)$, 
\item
$b$ is {\em strongly essential} if  $b$ is not boundary-parallel in $S_t$,
\item 
$b$ is {\em separating} if $b$ separates the page $S_t$ into two connected components.
\end{enumerate}
For a b-arc the conditions `{\em boundary parallel in $S_t$}' and `{\em non}-strongly essential' are equivalent. 
\end{definition}

\begin{definition}
An elliptic point $v$ is called {\em strongly essential} if every $b$-arc that ends at $v$ is strongly essential.
An open book foliation $\F(F)$ is called ({\em strongly}) {\em essential} if all the $b$-arcs are (strongly) essential. 
\end{definition}

The next lemma may be one of the most useful results about open book foliations, which claims that the existence of strongly essential elliptic points gives an estimate of the FDTC. 
See \cite[Section 5]{ik2} for further relationships between open book foliations  and the FDTC. 

\begin{lemma}\cite[Lemma 5.1]{ik2}
\label{lemma:estimate}
Let $v$ be an elliptic point of $\F(F)$ lying on a binding component $C \subset \partial S$. 
Assume that $v$ is strongly essential and there are no a-arcs around $v$.
Let $p$ $($resp. $n)$ be the number of positive $($resp. negative$)$ hyperbolic points that are joined with $v$ by a singular leaf. 
\begin{enumerate}
\item If $\sgn(v)= +1$ then $-n \leq c(\phi,C) \leq p.$
\item If $\sgn(v) = -1$ then $-p \leq c(\phi,C) \leq n.$
\end{enumerate} 
\end{lemma}

The embedding of $F$ near a hyperbolic point can be  described as follows: 
Recall that a hyperbolic point is a saddle tangency of a page and $F$. Consider a saddle-shaped subsurface of $F$ where leaves $l_1$ and $l_2$ (possibly $l_1=l_2$) as in Figure~\ref{fig:hyperbolic} are sitting on a page $S_t$.  
As $t$ increases (the page moves up) the leaves approach along a properly embedded arc $\gamma \subset S_t$ (dashed in Figure~\ref{fig:hyperbolic}) joining $l_{1}$ and $l_{2}$ and switch the configuration. 
See the passage in Figure~\ref{fig:hyperbolic}. 
We call $\gamma$ a {\em describing arc} of the hyperbolic point. 
Up to isotopy, $\gamma$ is uniquely determined and conversely $\gamma$ uniquely determines an embedding of the saddle. We often put the sign of a hyperbolic point near its describing arc. 
\begin{figure}[htbp]
%\ShowGrid
\SetLabels
(.1*.78) $\gamma$\\
(.36*.82) $F$\\
(.52*.55) $l_2$\\
(.05*.86) $l_1$\\
(.18*.86)   $l_2$\\
(.46*.59) $\gamma$\\
(.36*.55) $l_1$\\
(.58*.55)   $S_t$\\
\endSetLabels
\strut\AffixLabels{\includegraphics*[width=100mm]{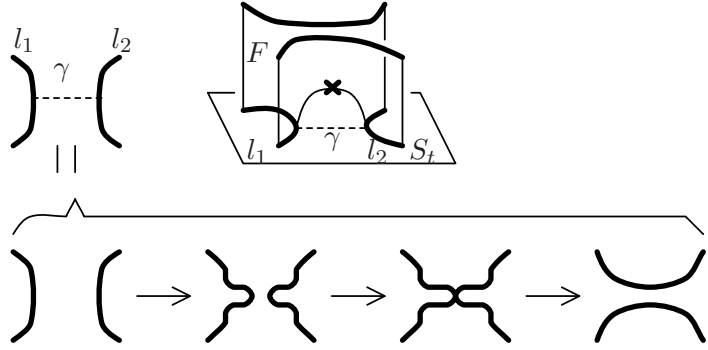}}
\caption{A describing arc (dashed) for a hyperbolic singularity.}
\label{fig:hyperbolic}
\end{figure}

The complement $\overline{ M_{(S,\phi)} \setminus S_{0}}$  of the page $S_0$ can be identified with $S\times[0,1]/\sim_{\partial}$, where $\sim_{\partial}$ is an equivalence relation $(x,t) \sim_{\partial} (x,s)$ for $x \in \partial S$ and $s,t \in [0,1]$. 
Let 
\begin{equation}\label{projection}
\mathcal{P}: (S\times[0,1]/\sim_{\partial}) \rightarrow S
\end{equation}
be the projection defined by $\mathcal P (x, t) = x$. 
To compare leaves in different pages we often use the projection $\mathcal{P}$. 
For example, by saying ``b-arcs $b \in S_t$ and $b' \in S_{t'}$ intersect'' we mean the arcs $\mathcal{P}(b)$ and $\mathcal{P}(b')$ intersect in $S$.

A {\em movie presentation} \cite{ik1-1} of $F$ describes how $F$ is embedded up to isotopy:  
Take $0=s_{0}<s_{1}<\cdots < s_{k} =1$ such that $S_{s_{i}}$ is a regular page and there exists exactly one hyperbolic point $h_{i}$ in each interval $(s_{i},s_{i+1})$. 
The sequence of slices $\{(S_{s_{i}}, S_{s_{i}} \cap F )\}$
with a describing arc of $h_{i}$ is called a movie presentation.  

\begin{example}
Let $(D^{2},id)$ be an open book decomposition of $S^{3}$. Consider a $2$-sphere $F$ embedded in $S^{3}$ as shown in Figure~\ref{fig:movie}-(a). Figure~\ref{fig:movie}-(b) depicts the entire picture of $\F(F)$ and Figure~\ref{fig:movie}-(c) is a movie presentation. 
\begin{figure}[htbp]
\begin{center}
\SetLabels
(0.0*0.98) (a)\\
(0.52*0.98) (b)\\
\endSetLabels
\strut\AffixLabels{\includegraphics*[scale=0.5, width=70mm]{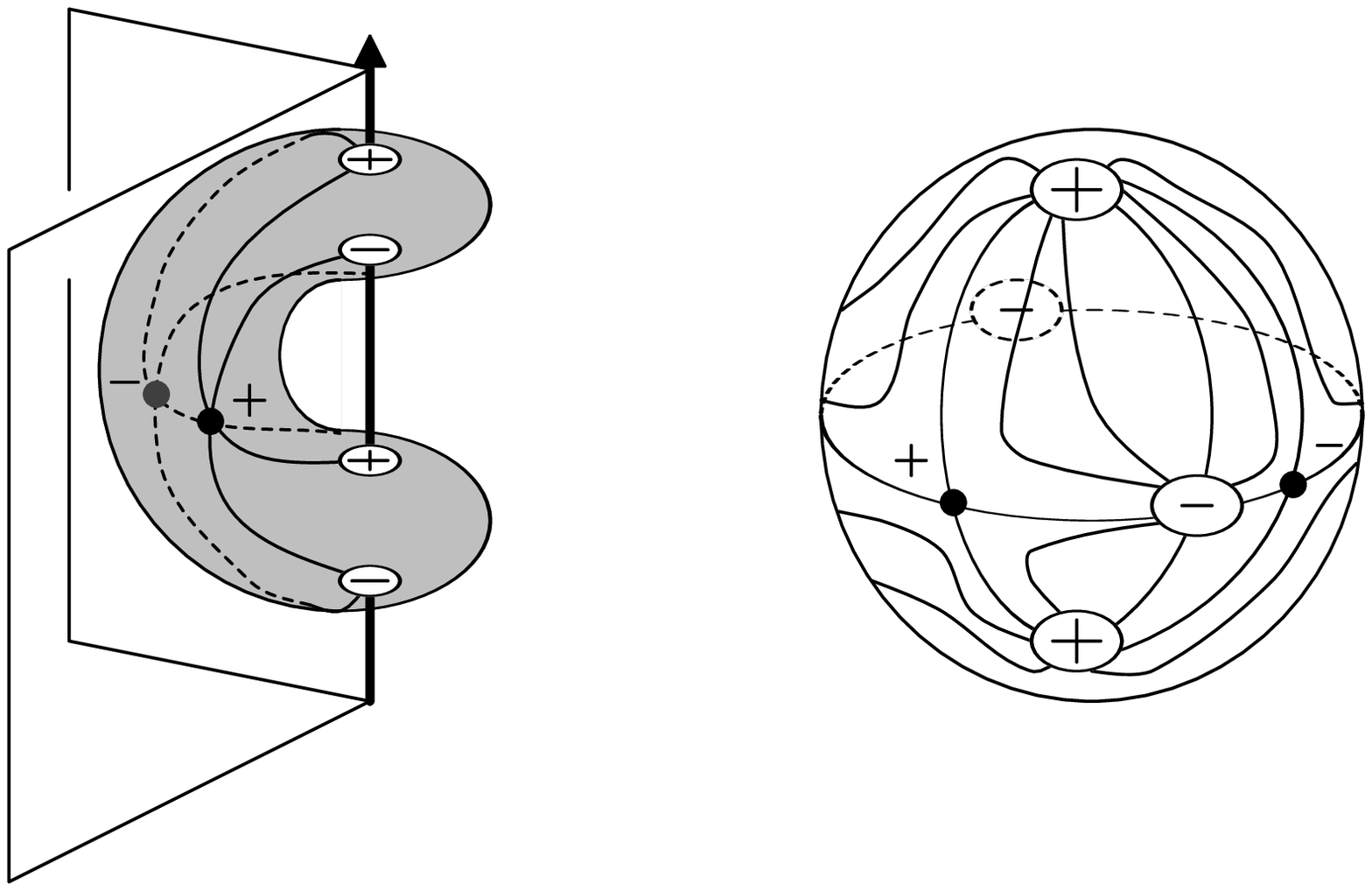}}
\SetLabels
(0.0*0.98) (c)\\
(0.12*0.3) $S_{0}$\\
(0.44*0.3) $S_{\frac{1}{2}}$\\
(0.74*0.3) $S_{1}$\\
(0.15*0.7) $-$\\
(0.51*0.62) $+$\\
(1.03*0.15) $id$\\
\endSetLabels
\strut\AffixLabels{\includegraphics*[scale=0.5, width=100mm]{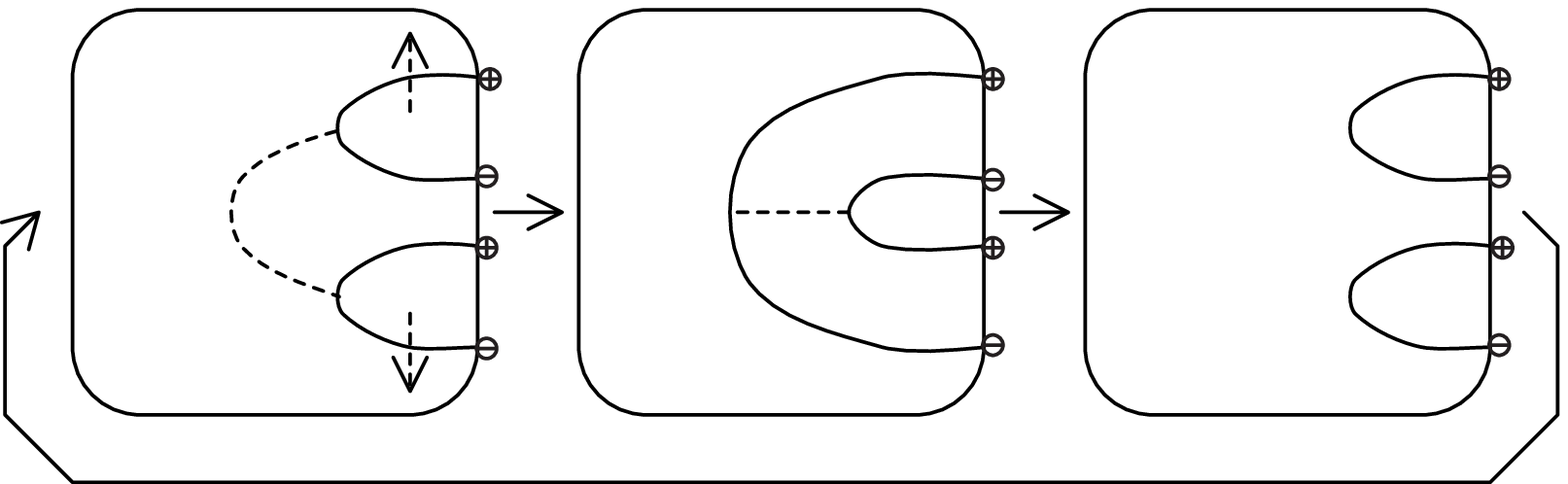}}
\caption{A movie presentation. 
Dashed arrows indicate normals $\vec{n_F}$ to $F$.}
\label{fig:movie}
\end{center}
\end{figure}
\end{example}

Both the surface $F$ and the ambient manifold $M$ are oriented. Let $\vec n_F$ be a positive normal to $F$. 
We orient each leaf of $\F(F)$, for both regular and singular, so that if we stand up on the positive side of $F$ and walk along a leaf in the positive direction then the positive side of the corresponding page $S_t$ of the open book is on our left. 
In other words, at a non-singular point $p$ on a leaf $l \subset (S_t \cap F)$ let $\vec n_S$ be a positive normal to $S_t$ then $X_{ob}= \vec n_S \times \vec n_F$ is a positive tangent to $l$. 
As a consequence, positive/negative elliptic points become sources/sinks of the vector field $X_{ob}$. 

Here is a useful fact about the sign of a hyperbolic point, its describing arc and $\vec n_F$. 

\begin{observation}\label{obs of sign} 
A hyperbolic point is positive (resp. negative) if and only if positive normals $\vec n_F$ point out of (resp. into) its describing arc.
See Figure~\ref{describing-arc}.
\begin{figure}[htbp]
\begin{center}
%\ShowGrid
\SetLabels
(0*.7)$\vec n_F$\\
(.2*.7)$\vec n_F$\\
(.1*.82)$+$\\
(.1*.25)$-$\\
\endSetLabels
\strut\AffixLabels{\includegraphics*[scale=0.5, width=110mm]{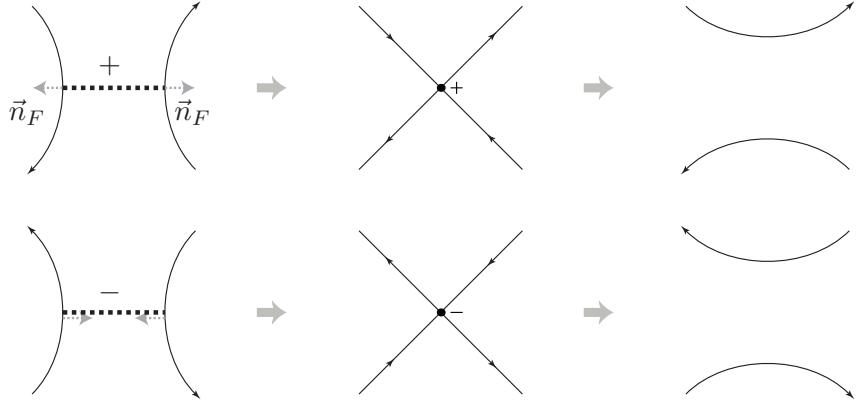}}
\caption{Observation~\ref{obs of sign}. Relation of the signs of hyperbolic points, describing arcs (thik dashed) and normal vectors $\vec n_F$ (dashed gray arrows).}
\label{describing-arc}
\end{center}
\end{figure}
\end{observation}

The {\em graph $G_{--}=G_{--}(F)$} of $\F(F)$ is a graph which consists of negative elliptic points, negative hyperbolic points and the unstable separatrices for negative hyperbolic points in $aa$- $ab$- and $bb$-tiles.
 The vertices of $G_{--}$ are the negative elliptic points in $ab$- and $bb$-tiles and the end points of the edges of $G_{--}$ that lie on $\partial F$, called {\em fake} vertices.
Similarly, the graph $G_{++}$ consists of positive elliptic points, positive hyperbolic points and the stable separatrices of positive hyperbolic points.

\begin{definition}\label{def:trans-ot-disc}\cite[Definition 4.1]{ik1-1}
An embedded disc $D \subset M_{(S, \phi)}$ whose boundary is a positively braided unknot is called a {\em transverse overtwisted disc} if 
\begin{enumerate}
\item $G_{--}$ is a connected tree with no fake vertices.
\item $G_{++}$ is homeomorphic to $S^1$.
\item $\F(D)$ contains no c-circles. 
\end{enumerate}
\end{definition}

As proved in \cite[Proposition 4.2, Corollary 4.6]{ik1-1} an open book $(S,\phi)$ supports an overtwisted contact structure if and only if $M_{(S,\phi)}$ contains a transverse overtwisted disc.

In \cite{ik3}, we study operations on open book foliations including b-arc foliation changes:  
\begin{theorem}[b-arc foliation change]\cite[Theorem 3.1, Proposition 3.2]{ik3}
\label{theorem:b-arcfolchange}
Assume that the open book foliation $\F(F)$ contains two tiles $R_1, R_2$ satisfying the following conditions {\bf(i)}--{\bf(iii)}, see Figure~\ref{fig:folchange}-(a): 
\begin{description}
\item[(i)] 
$R_i$ $(i=1,2)$ is either an $ab$-tile or a $bb$-tile.
\item[(ii)] 
$\sgn(R_{1})=\sgn(R_{2})=\e \in \{+1, -1\}.$
\item[(iii)] 
$R_{1}$ and $R_{2}$ are adjacent at exactly one leaf that is  a separating $b$-arc, $b$.
\end{description}
Then there is an ambient isotopy $\Phi_{\tau}:M \rightarrow M$ supported on $M \setminus B$ such that: 
\begin{enumerate}
\item
$F'=\Phi_{1}(F)$ admits an open book foliation $\F(F')$. If $\F(F)$ is essential then so is $\F(F')$. 
\item 
The region decomposition of $\F(F')$ contains regions $R'_{1}, R'_{2}$, see Figure~\ref{fig:folchange}-(b,c) such that: 
\begin{enumerate}
\item 
Each $R'_i$ is an $aa$-, $ab$-, or $bb$-tile.
\item 
$\sgn(R'_{1})=\sgn(R'_{2})=\e$ as in {\bf(ii)} above. 
\item 
$\Phi_1(R_1 \cup R_2)= R_1' \cup R_2'$.
\item 
$R'_{1}\cap R'_{2}$ is exactly one leaf, $l$, of type $a$ or $b$. 
\item
The numbers of the hyperbolic points connected to the elliptic points $v$ and $w$ by a singular leaf decrease both by one, though the total number of hyperbolic points remains the same. 
\end{enumerate}
\item 
$\Phi_t$ preserves the region decomposition of $F\setminus(R_{1} \cup R_{2})$. 
\item 
If $\partial F$ is non-empty then $\Phi_{t}(\partial F)$ is a closed braid with respect to  $(S, \phi)$ for all $t \in [0,1]$, i.e., $L=\partial F$ and $L' = \partial F'$ are braid isotopic. 
\end{enumerate}
\end{theorem}
\begin{figure}[htbp]
 \begin{center}
%\ShowGrid
\SetLabels
(.01*.73) (a)\\ 
(.2*.5)  $b$\\
(.1*.53) $\e$\\
(.35*.53) $\e$\\
(0.35*0.38)  $R_{1}$\\
(0.13*0.38)  $R_{2}$\\
(.22*.23) $w$\\
(.22*.74) $v$\\
(.55*.95) (b)\\
(.73*.67) $\e$\\
(.81*.85) $\e$\\
(0.92*0.8)  $R_{2}'$\\
(0.65*0.7)  $R_{1}'$\\
(.77*.75) $l$\\
(.82*.53) $w$\\
(.82*1) $v$\\
(.55*0.45) (c)\\
(.73*.31) $\e$\\
(.81*.14) $\e$\\
(0.93*0.2)  $R_{2}'$\\
(0.8*0.3)  $R_{1}'$\\
(.77*.23) $l$\\
(.82*0) $w$\\
(.82*.45) $v$\\
\endSetLabels
\strut\AffixLabels{\includegraphics*[scale=0.5, width=90mm]{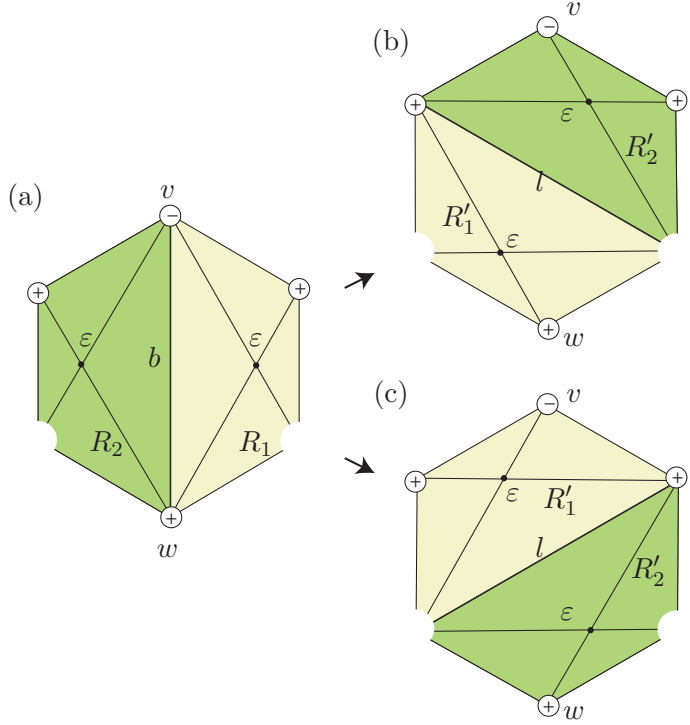}}
\caption{
Transitions (a)$\to$(b) and (a)$\to$(c) are $b$-arc foliation changes. 
At each corner hole a negative elliptic point (reps. the boundary of $F$) is placed if $R_i$ is a bb-tile (reps. ab-tile). }
\label{fig:folchange}
\end{center}
\end{figure}

\section{Outline of the proof of Theorem \ref{theorem:main}}
\label{sec:outline}

The rest of the paper is devoted to proving Theorem \ref{theorem:main}.
In this section we give an overview of the proof.
Assume that a planar open book $(S,\phi)$ supports an overtwisted contact structure. 
We start with an arbitrary transverse overtwisted disc $D$ in $M_{(S,\phi)}$ and introduce a {\em complexity} of $D$ which measures how far $D$ is from having the property {\bf (SE1)}. 
We construct a new transverse overtwisted disc $D'$ whose complexity is less than that of $D$. 
By standard induction on the complexity we finish the proof.

Here is a more detailed outline. 

In Section \ref{sec:niceposition}, as an intermediate step we construct from $D$  an embedded disc $D_*$ by replacing a boundary parallel (i.e., non-strongly essential) b-arc of $D$ with a non-essential b-arc. 

In Section \ref{sec:foliation}, we study the open book foliation of $D_*$ and show how $\F(D_*)$ is related to $\F(D)$ (cf. Figure \ref{fig:blowup}). 

In Section \ref{sec:complexity}, we construct a transverse overtwisted disc $D'$ from $D_*$. After studying basic properties of $D'$ we define a complexity of a transverse overtwisted disc and prove that $D'$ has smaller complexity than that of the original $D$. 

In Section \ref{sec:proof}, we complete the proof of Theorem \ref{theorem:main}.

\section{Movie presentation of the intermediate disc $D_*$}
\label{sec:niceposition}

Let $(S, \phi)$ be a planar open book supporting an overtwisted contact structure. 
Let $D$ be a transverse overtwisted disk in $M_{(S, \phi)}$. 
Assume that $D$ does not satisfy the property {\bf (SE1)} stated in Theorem \ref{theorem:main}.

In this section we construct a disc $D_*$ from  $D$ mentioned in Section~\ref{sec:outline}. This is a crucial intermediate step to find an overtwisted disc with {\bf (SE1)}. We do this by converting a boundary parallel b-arc into a non-essential b-arc. 

By \cite[Theorem 3.2]{ik2} we may assume that $\F(D)$ is essential.

Let $v \in G_{--}(D)$ be a valence one vertex. That is, $v$ is a negative elliptic point of $\F(D)$ and the number of negative hyperbolic points connected to $v$ by a singular leaf is one. Let $m$ be the number of positive hyperbolic points connected to $v$ by a singular leaf. 
Call the positive (resp. negative) hyperbolic points $h_1,\dots,h_m$ (resp. $h_-$). 
At $v$ one bb-tile $R_-$ containing $h_-$ and $m$ ab-tiles $R_1,\dots,R_m$ containing $h_1,\dots,h_m$ respectively meet. 
Let $\Omega_0,\ldots,\Omega_m$ denote the positive elliptic points which are connected to $v$ by a b-arc such that $\Omega_i \in \partial R_i \cap \partial R_{i+1}$ and $\Omega_0, \Omega_m \in \partial R_-$.   
For $t \in[0,1]$ we denote the b-arc in the page $S_{t}$ that ends at $v$ by $b_{t}$, see Figure~\ref{fig:nearv}-(a). 
\begin{figure}[htbp]
\begin{center}
%\ShowGrid
\SetLabels
(0.02*0.98) (a)\\
(0.55*0.98) (b)\\
(0.2*0.5) \Large$v$\\
(0.31*0.25) $\Omega_0$\\
(0.13*0.3) $\Omega_1$\\
(0.105*0.7) $\Omega_{m\!-\!1}$\\
(0.36*0.75) $\Omega_{\!m}$\\
(0.22*0.28) $h_1$\\
(0.22*0.73) $h_{\!m}$\\
(0.36*0.46) $h_{-}$\\
(0.62*0.13) \Large $v$\\
(0.62*0.85) $\Omega_0$\\
(0.92*0.14) $S_0$\\
(0.89*0.42) $\Delta$\\
(0.63*0.33) $x_n$\\
(0.63*0.62) $x_2$\\
(0.63*0.7) $x_1$\\
(.98*.5) $b_0$\\
(0.8*0.56) {\rotatebox{-180}{\LARGE $A^\Box$}}\\
%(0.75*0.46) {\ \rotatebox{-180}{$\Box$}}\\
\endSetLabels
\strut\AffixLabels{\includegraphics*[width=110mm]{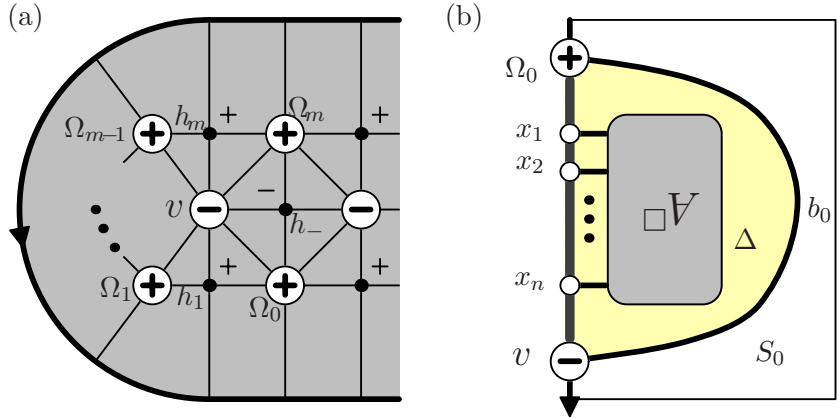}}
\caption{
(a) $\F(D)$ near the valence one vertex $v$. 
(b) The leaf box $A^{\Box}$ (Definition~\ref{definition:leafbox}) in the page $S_{0}$ represents the set of leaves bounded by $b_0$. The label $A^{\Box}$ is  upside-down because the orientation of the binding is downward.
}
\label{fig:nearv}
\end{center}
\end{figure}

We may assume that the page $S_{0}= \phi(S_{1})$ is a regular fiber. 
Denote the singular fiber that contains $h_{i}$ by $S_{t_i}$ where $t_{i} \in (0,1)$.
There exists a small $\varepsilon>0$ such that: 
\begin{itemize}
\item $h_{-} \in S_{1-\varepsilon}$.
\item For any distinct singular fibers $S_{t}$ and $S_{t'}$ we have $|t-t'|>2\varepsilon$. 
\item The family $\{S_{t}\: | \:t_{i}-\varepsilon \leq t \leq t_{i}+\varepsilon \}$ contains exactly one hyperbolic point which is  $h_{i}$.
\item $0 < t_1 <t_2 < \cdots < t_m < 1$.  
\end{itemize}

\begin{lemma}
\label{lemma:removesep}
With some perturbation of $D$ we may assume that: 
\begin{description}
\item[(P1)] The b-arc $b_t$ is non-separating for all $t \in (t_1,t_m)$ with $t \neq t_1, \ldots, t_m$. 
\end{description}
\end{lemma}

\begin{proof}
Assume that $b_t$ is separating for some $t \in (t_i,t_{i+1})$. 
Then the ab-tiles $R_{i}$ and $R_{i+1}$ meet along a separating b-arc. 
Since $\sgn(R_i)=\sgn(R_{i+1})$, applying a b-arc foliation change (Theorem \ref{theorem:b-arcfolchange}) the region $R_i \cup R_{i+1}$ is replaced by the union of one positive ab-tile and one positive aa-tile. 
See the passage (a) in Figure~\ref{fig:removesep}. 
The new aa-tile can be eliminated by destabilizing the closed braid $\partial D$, see Figure~\ref{fig:removesep}-(b). 
The resulting disc is a transverse overtwisted disc.

As a consequence, the family of separating b-arcs $\{b_t \: |\: t \in (t_i,t_{i+1})\}$ disappears and the number of positive elliptic points connected to $v$ decreases by one. 
\end{proof}
\begin{figure}[htbp]
 \begin{center}
%\ShowGrid
\SetLabels
(0*0.96) {\small separating b-arc}\\
(0.25*0.97) $v$\\
(0.42*0.86) $\Omega_{i-1}$\\
(0.24*0.85) $\Omega_{i}$\\
(0.07*0.86) $\Omega_{i+1}$\\
(0.42*0.24) $\Omega_{i-1}$\\
(0.25*0.38) $v$\\
(0.07*0.24) $\Omega_{i+1}$\\
(0.33*0.75) $h_{i}$\\
(0.12*0.75) $h_{i+1}$\\
(0.24*0.16) $h'$\\
(0.68*0.9) $h'$\\
(0.68*0.3) $h'$\\
(0.9*0.75) $h''$\\
(0.5*0.82) (a)\\
(0.82*0.52) (b)\\
(0.5*0.22) (c)\\
(0.26*0.65)  $R_{i}$\\
(0.16*0.65) $R_{i+1}$\\
\endSetLabels
\strut\AffixLabels{\includegraphics*[width=100mm]{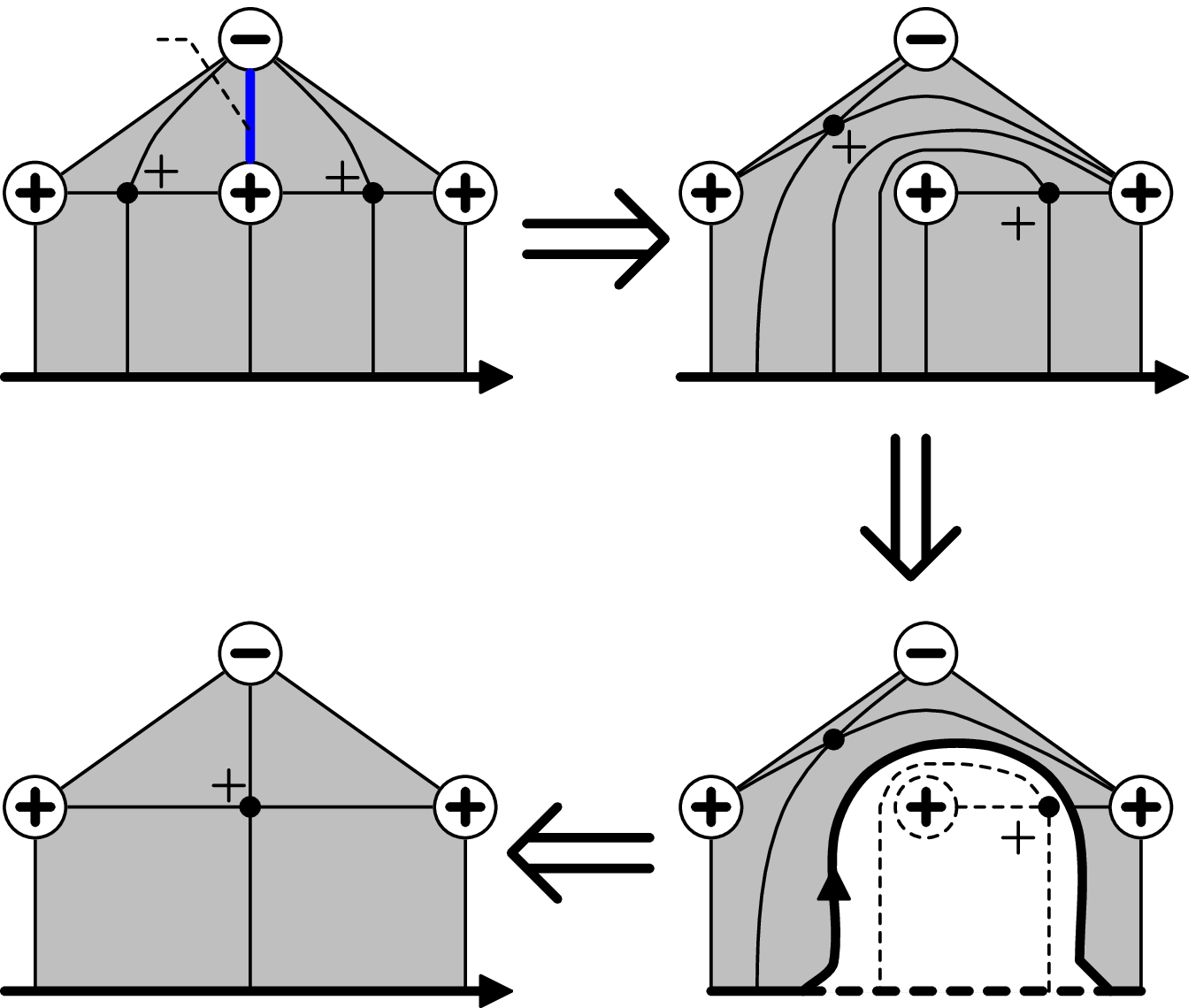}}
\caption{(Lemma \ref{lemma:removesep}): 
If the bb-tiles $R_{i}$ and $R_{i+1}$ are adjacent at a separating b-arc then we apply: \\ 
(a) b-arc foliation change. \\
(b) destabilization of the closed braid $\partial D$. \\
(c) rescaling the open book foliation.}
 \label{fig:removesep}
  \end{center}
\end{figure}

From now we assume that $v \in \F(D)$ is a non-strongly essential,  valence one vertex satisfying the property {\bf (P1)}.
This means that either $b_0$ or $b_{1-2\e}$ is boundary parallel (i.e., non-strongly essential). 
We may assume that $b_0$ is boundary parallel in $S_0$ and cobounds a disc $\Delta \subset S_{0}$ with a binding component. 
(The other case can be treated similarly.) 
Let $\partial' \Delta := \partial \Delta \setminus b_0$. 
Since $\F(D)$ is essential $\Int(\Delta)$ intersects $D$. Hence $\partial' \Delta$ contains elliptic points, $x_{1},\ldots,x_{n}$.

\begin{definition}
\label{definition:leafbox}
Let $A := \Int(\Delta) \cap D \subset S_{0}$ be the set of leaves contained in $\Delta$. 
The end points of $A$ are $x_1,\dots,x_n$. 
In movie presentations, $A$ is represented  by a gray box labeled $A^\Box$ as depicted in Figure \ref{fig:nearv}-(b). 
We call it a \emph{leaf box} for $A$.
\end{definition}

The following is a key observation where the planar assumption of $S$ plays a crucial role.

\begin{lemma}
\label{lemma:key}
Assume that $S$ is planar. Let $v$ be a non-strongly essential, valence one vertex satisfying {\bf(P1)} such that $b_0$ is boundary parallel in $S_0$.
Then $D$ satisfies the property 
\begin{description}
\item[(P2)] $\Omega_i \not \in \partial'\Delta$ for all $i=1,\ldots,m$. That is, $\Omega_i \neq x_1, \dots, x_n$ for all  $i=1,\ldots,m$.
\end{description}
\end{lemma}

\begin{proof} 
Assume that $\Omega_{i} \in \partial' \Delta$ for some $i=1,\ldots,m-1$. 
Since $S$ is planar the b-arc $b_{t_{i}+\varepsilon}$ connecting $v$ and $\Omega_{i}$ is separating, which contradicts {\bf (P1)}.

Next we show that $\Omega_{m} \not \in \partial' \Delta$. 
Using the projection $\mathcal P$ in (\ref{projection}) we compare objects in different pages. 
Since the family $\{b_t \ | \ 1-2\e \leq t \leq 1\sim 0\}$ contains only one hyperbolic point which is $h_-$, the interiors of $\mathcal P(b_{1-2\e})$ and $\mathcal P(b_0)$ have zero geometric intersection. 
Therefore, if $\Omega_{m} \in \partial' \Delta$ then the b-arc $\mathcal P(b_{1-2\e})$ must be included in $\mathcal P(\Delta)$.
However, since $h_{-}$ is a negative hyperbolic point,  Observation~\ref{obs of sign} implies that $\mathcal P(b_{1-2\e})$ must lie on the right of $\mathcal P(b_0)$ near $v$ (i.e., $\mathcal P(b_{1-2\e})$ is outside $\mathcal P(\Delta)$), which is a contradiction.
\end{proof}

Now by modifying $D$ we construct a new embedded disc $D_*$.
The elliptic points $v, \Omega_0,\dots,\Omega_m$ and the b-arc $b_0 \subset S_0$ are kept the same under the construction.  
However $b_{0}$ of the new disc $D_*$ will become non-essential.   %The definition of $D_*$ is given by movie presentations. 
%{\color{red} First of all, we put the disc $D$ so that $\Int(\Delta)\times[0,t_1)$ contains no hyperbolic points}. 
We do this by moving the the set of leaves $A$ out of $\Delta$ at the cost of introducing new singular points. 
The disc $D_*$ may not be a transverse overtwisted disc, but is similar to a transverse overtwisted disc in the sense that $\partial D_*$ violates the Bennequin-Eliashberg inequality. 

%One can check that the resulting surface $D_*$ is indeed a disc by checking that:
%\begin{itemize}
%\item Both $D_*$ and $\partial D_*$ is connected.
%\item The Euler characteristic of $D_*$ is minus one. Note that teh Euler chacarteristic is computed by counting the number of hyperbolic and elliptic points.
%\end{itemize}

%However, in the next section, we will observe that the resulting surface is indeed a disc, by looking at how the open book foliation of $D_*$ is related to the original open book foliation of $D$, so at this moment we do not need to check $D_*$ is a disc.\\

The definition of $D_*$ is given in four steps of movie presentations. \\

\textbf{Step 1}: Movie for $t \in [0,t_1+\varepsilon]$. \\

See Figure \ref{fig:step1}, where the left column depicts a  movie presentation of $D$ for $t \in [0,t_1+\e]$ near $\Delta\cup b_{t_1+\e}$. We can identify $\{ b_t \ | \ 0 \leq t < t_1 \} \cong (\partial \Delta \setminus \partial'\Delta) \times [0, t_1)$ 
and may assume the region $\Delta \times [0,t_{1})$ contains no hyperbolic points. 
Thus $D \cap (\Delta \times [0, t_{1})) \cong A \times [0, t_{1})$.

A movie presentation of $D_*$ for $t \in [0,t_1+\e]$ near $\Delta\cup b_{t_1+\e}$ is given in the right column of Figure~\ref{fig:step1}. 
There are no leaves in $\Delta \times \{t\}$ for $t\in[0,t_1)$, i.e., no elliptic points on $\partial'\Delta$.  
Instead, for each page new elliptic points $x_{1}^{(1)},\ldots,x_{n}^{(1)}$  with $\sgn(x^{(1)}_j)=\sgn(x_j)$ and a copy of the leaves $A$ are placed on the right-hand side of $\Omega_1$. 
\begin{figure}[htbp]
\begin{center}
%\ShowGrid
\SetLabels
(.27*.97) {\scriptsize $0\leq t < t_1-\varepsilon$}\\
(0.035*0.95) $\Omega_0$\\
(0.13*0.88) {\rotatebox{-180}{$A^\Box$}}\\
%(0.14*0.88) {\huge \rotatebox{-180}{$A$}}\\
%(0.11*0.84) {\scriptsize \rotatebox{-180}{$\Box$}}\\
(0.04*0.9) $x_{1}$\\
(0.04*0.81) $x_{n}$\\
(0.04*0.74) {\large $v$}\\
(0.42*0.73) $\Omega_{1}$\\
%%%%%%%%%%%%%%%%%%%%%%%%%%%%%%%%%%%%%%%%%%%
(.27*.61) {\scriptsize $t = t_1-\varepsilon$}\\
(0.035*0.61) $\Omega_{0}$\\
(0.13*0.54) {\rotatebox{-180}{$A^\Box$}}\\
%(0.14*0.54) {\huge \rotatebox{-180}{$A$}}\\
%(0.11*0.5) {\scriptsize \rotatebox{-180}{\sf box}}\\
(0.04*0.56) $x_{1}$\\
(0.04*0.47) $x_{n}$\\
(0.04*0.4) {\large $v$}\\
(0.42*0.39) $\Omega_{1}$\\
(0.2*0.39) $+$\\
(0.25*0.36) $h_{1}$\\
%%%%%%%%%%%%%%%%%%%%%%%%%%%%%%%%%%%%%%%%
(.27*.28) {\scriptsize $t_1< t \leq t_1 + \e$}\\
(0.035*0.27) $\Omega_{0}$\\
(0.13*0.2) {\rotatebox{-180}{$A^\Box$}}\\
%(0.14*0.2) {\huge \rotatebox{-180}{$A$}}\\
%(0.11*0.16) {\scriptsize \rotatebox{-180}{\sf box}}\\
(0.04*0.22) $x_{1}$\\
(0.04*0.14) $x_{n}$\\
(0.04*0.06) {\large $v$}\\
(0.42*0.05) $\Omega_{1}$\\
(.27*.09) $b_{t_1+\e}$\\
%%%%%%%%%%%%%%%%%%%%%%%%%%%%%%%%%%%%
(0.58*0.95) $\Omega_{0}$\\
(.8*.97) {\scriptsize $0\leq t < t_1-\varepsilon$}\\
(.88*.85) $A^\Box$\\
%(0.86*0.82) {\huge $A$}\\
%(0.89*0.86) {\scriptsize {\sf box}}\\
(0.98*0.9) $x_{n}^{(1)}$\\
(0.98*0.81) $x_{1}^{(1)}$\\
(0.58*0.74) {\large $v$}\\
(0.98*0.73) $\Omega_{1}$\\
%%%%%%%%%%%%%%%%%%%%%%%%%%%%%%%%%%%%
(0.58*0.61) $\Omega_{0}$\\
(.8*.61) {\scriptsize $t = t_1-\varepsilon$}\\
(0.58*0.4) {\large $v$}\\
(0.98*0.39) $\Omega_{1}$\\
(0.98*0.56) $x_{n}^{(1)}$\\
(0.98*0.47) $x_{1}^{(1)}$\\
(.88*.5) $A^\Box$\\
%(0.86*0.48) {\huge $A$}\\
%(0.89*0.52) {\scriptsize {\sf box}}\\
(0.75*0.39) $+$\\
(0.8*0.36) $h_{1}$\\
%%%%%%%%%%%%%%%%%%%%%%%%%%%%%%%%%%%%
(0.58*0.27) $\Omega_{0}$\\
(.8*.28) {\scriptsize $t_1< t \leq t_1 + \e$}\\
(0.58*0.06) {\large $v$}\\
(0.98*0.05) $\Omega_{1}$\\
(0.98*0.22) $x_{n}^{(1)}$\\
(0.98*0.13) $x_{1}^{(1)}$\\
(.88*.16) $A^\Box$\\
%(0.86*0.14) {\huge $A$}\\
%(0.89*0.18) {\scriptsize {\sf box}}\\
(.75*.09) $b_{t_1+\e}$\\
%%%%%%%%%%%%%%%%%%%%%%%%%%%%%%%%%%%%%
(0.25*-0.04) (Movie of $D$)\\
(0.75*-0.04) (Movie of $D_*$)\\
\endSetLabels
\strut\AffixLabels{\includegraphics*[width=135mm]{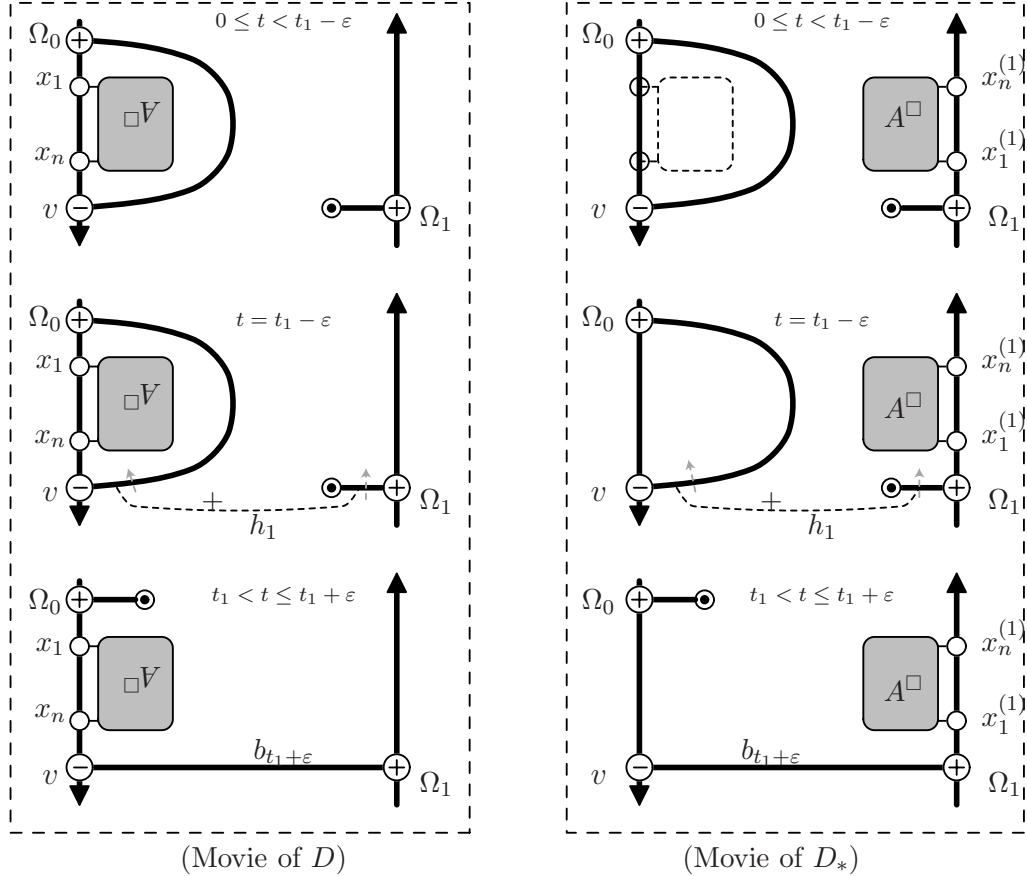}}
\vspace{0.3cm}
\caption{{\bf Step 1}: The movie presentations of $D$ and $D_*$ for $t\in [0,t_{1}+\varepsilon]$. Dashed gray arrows in the middle row indicate normals $\vec n_D$ and $\vec n_{D_*}$.}
\label{fig:step1}
\end{center}
\end{figure}

Away from the neighborhood of $\Delta\cup b_{t_1+\e}$ the disks $D$ and $D_*$ have the same movie presentation, except that we put elliptic points $x_1^{(i)},\dots,x_n^{(i)}$ with $\sgn(x^{(i)}_j)=\sgn(x_j)$ and a copy of $A$ on the right-hand side of $\Omega_i$ for each $i=2,\ldots,m$, see Figure~\ref{fig:step1-2}. 
These copies of $A$ will be used in {\bf Step 3} to ``switch legs''.
The property {\bf(P2)} guarantees that $x_{1}^{(i)},\ldots,x_{n}^{(i)} \not\in \partial' \Delta$ for all $i=1,\ldots,m$. 
\begin{figure}[htbp]
\begin{center}
%\ShowGrid
\SetLabels
(.2*.16) $\Omega_i$\\
(1.04*.16) $\Omega_i$\\
(1.05*.32) $x_1^{(i)}$\\
(1.05*.5) $x_2^{(i)}$\\
(1.05*.72) $x_n^{(i)}$\\
(.86*.55) $A^\Box$\\
(0*.9) $D\cap S_t$\\
(.7*.9) $D_* \cap S_t$\\
\endSetLabels
\strut\AffixLabels{\includegraphics*[width=80mm]{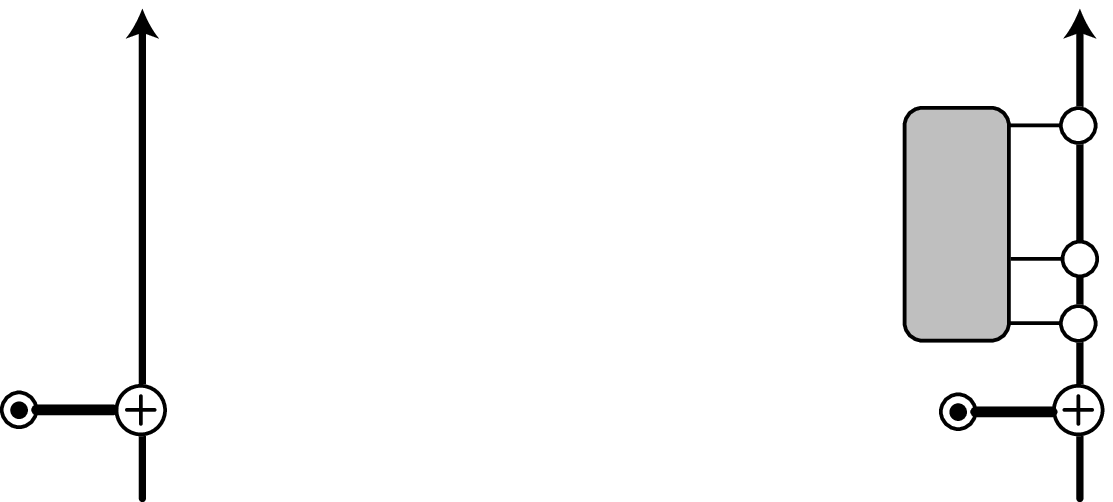}}
\caption{{\bf Step 1}: 
In the page $S_t$ for $t\in [0, t_1+\e]$ we put elliptic points $x_1^{(i)},\ldots, x_n^{(i)}$ and a copy of $A$ on the right-hand side of $\Omega_i$ for each $i=1,\dots,m$.}
\label{fig:step1-2}
\end{center}
\end{figure}
%
%{$ $} \\

\textbf{Step 2}: Movie for $t \in [t_{i}+\e, t_{i+1}-\e)$ $(i=1,2,\ldots,m-1)$ and $[t_{m}+\varepsilon,1-2\e)$.\\

The left sketch of Figure \ref{fig:slideleg} depicts the slice $D \cap S_t$ for $t \in [t_{i}+\e, t_{i+1}-\e)$ near $\partial'\Delta \cup b_t$.  
For $t \in [0,1]$ let $X_t \subset S_t$ be the set of leaves connected to $\partial'\Delta$ (i.e., end at the elliptic points $x_1,\ldots,x_n$). 
For $t \in [0,t_1]$ the set $X_t$ is identical to $A$. 
For $t\in(t_1, 1]$ the set $X_t$ may not be the same as $A$ because the leaf from $x_i$ in the page $S_t$ may be a singular leaf and contain a hyperbolic point. 
In movie presentations  $X_t$ is denoted by a leaf box labeled $X^\Box$ or $X_t^\Box$ (we often omit the subscript $t$ for simplicity).

By {\em legs} of $X$ we mean subset of $X$ near $x_1,\ldots,x_n$ ($n$ arcs coming out of the leaf box for $X$).  
The points $x_1,\ldots,x_n$ are called {\em feet} of the legs.
We construct $D_*$ for the interval $[t_{i}+\e, t_{i+1}-\e)$ by sliding the legs of  $X_t \subset S_t$ along the b-arc $b_t$, see the right sketch of Figure \ref{fig:slideleg}. 
This can be made possible due to {\bf (P2)}.
\begin{figure}[htbp]
 \begin{center}
%\ShowGrid
 \SetLabels
(0.05*0.8) $\Omega_{0}$\\
(0.56*0.8) $\Omega_{0}$\\
(0.22*0.56) $X_t^\Box$\\%{\huge $X$}\\
%(0.26*0.64) {\sf \scriptsize box}\\
(0.22*0.26) {\large $b_t$}\\
(0.8*0.12) {\large $b_t$}\\
(0.72*0.56) $X_t^\Box$\\%{\huge $X$}\\
%(0.76*0.64) $\Box$\\%{\sf \scriptsize box}\\
%(0.52*0.5) {\huge $S_t$}\\
(0.05*0.65) $x_{1}$\\
(0.05*0.5) $x_{n}$\\
(0.95*0.5) $x_{n}^{(i)}$\\
(0.95*0.3) $x_{1}^{(i)}$\\
(0.05*0.17) {$v$}\\
(0.55*0.17) {$v$}\\
(0.46*0.2) $\Omega_{i}$\\
(0.95*0.2) $\Omega_{i}$\\
(0.25*.02) (Slice $D \cap S_t$)\\
(0.75*.02) (Slice $D_*\cap S_t$)\\
\endSetLabels
\strut\AffixLabels{\includegraphics*[width=130mm]{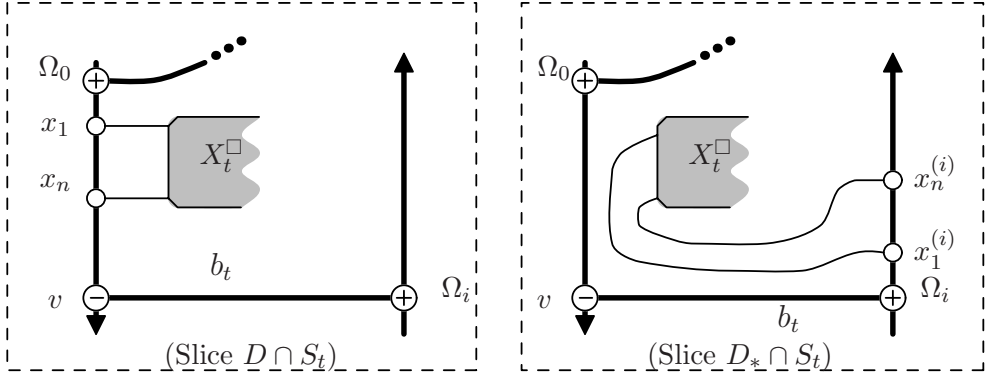}}
\vspace{0.3cm}
\caption{{\bf Step 2}: 
The slices $D_*\cap S_t$ and $D_*\cap S_t$ for $t\in [t_i+\e, t_{i+1}-\e)$.} 
\label{fig:slideleg}
\end{center}
\end{figure}

If for some $t^*\in [t_{i}+\varepsilon ,t_{i+1}-\varepsilon)$ the set of leaves $X_{t^*}$ contains a hyperbolic point $h^*$ then we use the same describing arc of $h^*$ (dashed arc in Figure \ref{fig:step2}) to define a hyperbolic point of the same sign. 
We can do this because the describing arc  does not intersect the b-arc $b_{t^*-\delta}$ where $\delta\ll\e$.
\begin{figure}[htbp]
\begin{center}
%\ShowGrid
\SetLabels
(.35*.95) $\scriptstyle{(t=t^*-\delta)}$\\
(.8*.95) $\scriptstyle{(t=t^*-\delta)}$\\
(.35*.43) $\scriptstyle{(t=t^*+\delta)}$\\
(.8*.43) $\scriptstyle{(t=t^*+\delta)}$\\
(0.07*0.92) $\Omega_{0}$\\
(0.54*0.92) $\Omega_{0}$\\
%(0.52*0.75) {\huge $S_{t^{*}-\delta}$}\\
(0.07*0.85) $x_{1}$\\
(0.07*0.76) $x_{j}$\\
(0.07*0.66) $x_{n}$\\
(0.08*0.59) $v$\\
(0.55*0.59) $v$\\
(0.46*0.6) $\Omega_{i}$\\
(0.95*0.85) $x_{n}^{(i)}$\\
(0.95*0.75) $x_{j}^{(i)}$\\
(0.95*0.66) $x_{1}^{(i)}$\\
(0.95*0.6) $\Omega_{i}$\\
(0.27*0.86) $h^{*}$\\
%(0.2*0.78) $\pm$\\
(0.75*0.86) $h^{*}$\\
(0.71*0.78) $\pm$\\
(.25*.56) $b_{t^*-\delta}$\\
(.25*.04) $b_{t^*+\delta}$\\
%%%%%%%%%%%%%%%%%%%%%%%%%
(0.07*0.4) $\Omega_{0}$\\
(0.54*0.4) $\Omega_{0}$\\
%(0.52*0.23) {\huge $S_{t^{*}+\delta}$}\\
(0.07*0.33) $x_{1}$\\
(0.07*0.23) $x_{j}$\\
(0.07*0.14) $x_{n}$\\
(0.08*0.07) $v$\\
(0.55*0.07) $v$\\
(0.46*0.1) $\Omega_{i}$\\
(0.95*0.33) $x_{n}^{(i)}$\\
(0.95*0.25) $x_{j}^{(i)}$\\
(0.95*0.15) $x_{1}^{(i)}$\\
(0.95*0.1) $\Omega_{i}$\\
(0.25*-0.04) (Movie of $D$)\\
(0.75*-0.04) (Movie of $D_*$)\\
(.75*.56) $b_{t^*-\delta}$\\
(.75*.04) $b_{t^*+\delta}$\\
\endSetLabels
\strut\AffixLabels{\includegraphics*[width=130mm]{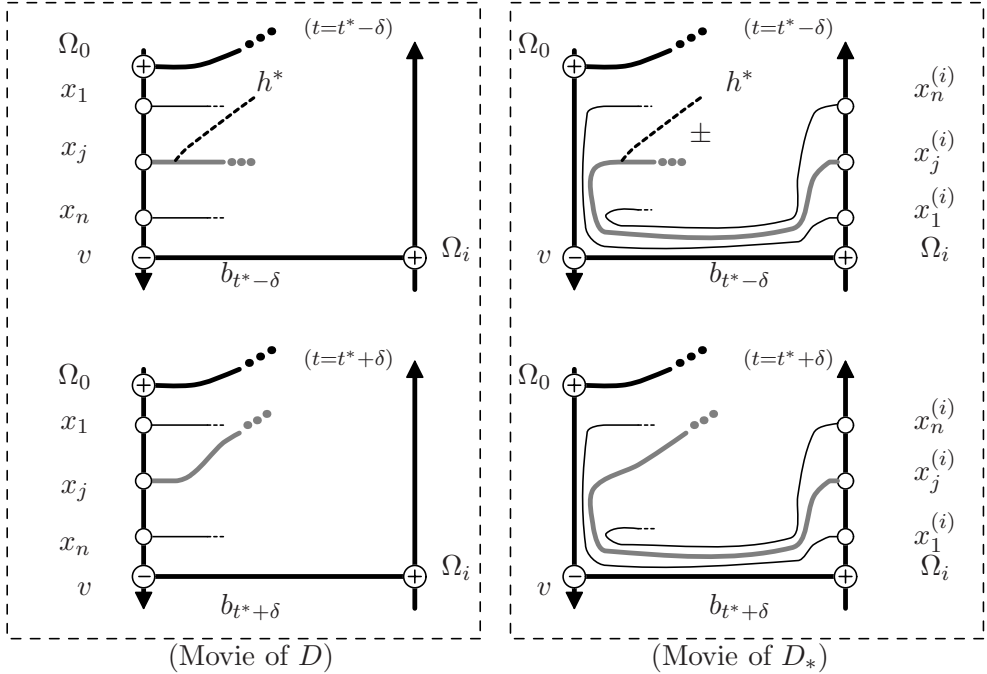}}
\vspace{0.3cm}
\caption{{\bf Step 2}: Treatment of a hyperbolic point where $\delta \ll \e$. } 
\label{fig:step2}
\end{center}
\end{figure}

Away from the neighborhood of $\partial'\Delta\cup b_{t}$ we define $D_*$ so that the movie presentations of $D$ and $D_*$ are identical, except that on the right-hand side of $\Omega_k$ for each $k\neq i$ the elliptic points $x_1^{(k)}, \ldots, x_n^{(k)}$ and a copy of $A$ are placed (see Figure~\ref{fig:step2-2}).
\begin{figure}[htbp]
\begin{center}
%\ShowGrid
\SetLabels
(.2*.16) $\Omega_k$\\
(1.04*.16) $\Omega_k$\\
(1.05*.32) $x_1^{(k)}$\\
(1.05*.5) $x_2^{(k)}$\\
(1.05*.72) $x_n^{(k)}$\\
(.85*.55) $A^\Box$\\
(0*.9) $D\cap S_t$\\
(.7*.9) $D_* \cap S_t$\\
\endSetLabels
\strut\AffixLabels{\includegraphics*[width=80mm]{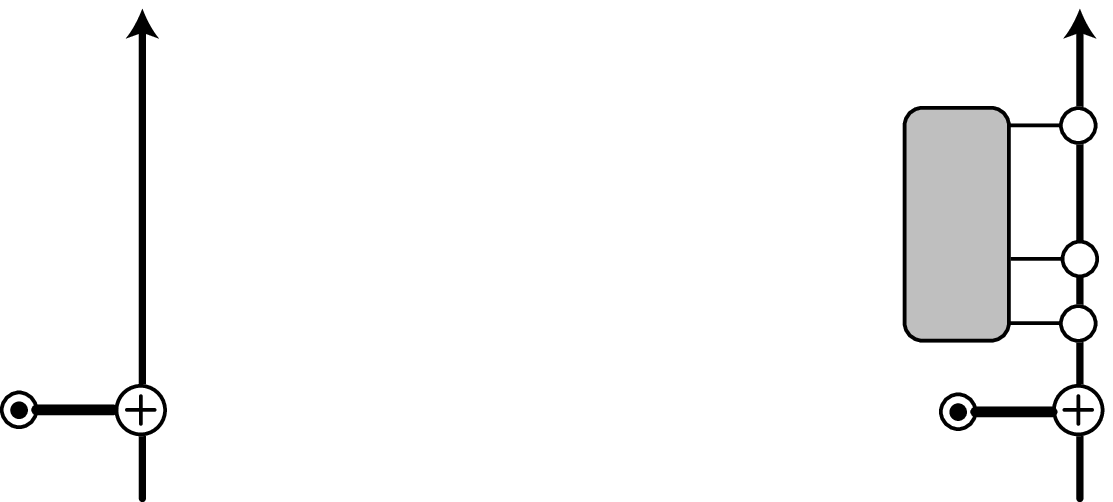}}
\caption{{\bf Step 2}: 
In the page $S_t$ for $t\in [t_i+\e, t_{i+1}-\e ]$ we put a copy of $A$ on the right-hand side of $\Omega_k$ for each $k \neq i$. \\
{\bf Step 3}:
In the page $S_t$ for $t\in [t_i-\e, t_i+\e ]$ we put a copy of $A$ on the right-hand side of $\Omega_k$ for each $k \neq i-1$, $i$.\\
{\bf Step 4}: 
In the page $S_t$ for $t\in [1-2\e, 1]$ we put a copy of $A$ on the right-hand side of $\Omega_k$ for each $k \neq m$.}
\label{fig:step2-2}
\end{center}
\end{figure}
\\

\textbf{Step 3}: Movie for $t\in [t_{i}-\varepsilon,t_{i}+\varepsilon]$ where $i=2,\ldots,m$.\\

On each page $S_t$ we put a copy of $A$ on the right-hand side of $\Omega_k$ for each $k \neq i-1$, $i$ as depicted in Figure~\ref{fig:step2-2}.

The construction of $D_*$ near $\Omega_{i-1}$ and $\Omega_i$ is summarized in Figure~\ref{fig:step3result}. 
The detail is described in two sub-steps. 
Since $h_{t_i}$ is the only hyperbolic point of $\F(D)$ in the interval $[t_{i}-\e,t_{i}+\e]$, we have $X_t \cong X_{t_i-\e}$ for all $t \in [t_{i}-\e,t_{i}+\e]$. 
We may denote $X_t$ simply by $X$. 
\begin{figure}[htbp]
 \begin{center}
%\ShowGrid
 \SetLabels
(0.07*0.91) $\Omega_{0}$\\
(0.56*0.91) $\Omega_{0}$\\
(.2*.1) \scriptsize $t=t_i-\e$\\
%(0.18*0.1) %{\Large $S_{t_{i}-\varepsilon}$}\\
(.67*.1) \scriptsize $t=t_i+\e$\\
%(0.66*0.1) {\Large $S_{t_{i}+\varepsilon}$}\\
(0.17*0.71) $X^\Box$\\%{\Large $X^{\sf box}$}\\
(0.67*0.71) $X^\Box$\\%{\Large $X^{\sf box}$}\\
(0.84*0.61) $A^\Box$\\%{\Large $A$}\\
%(0.85*0.65) {\sf \scriptsize box}\\
(0.33*0.22) $A^\Box$\\%{\Large $A$}\\
%(0.36*0.25) {\sf \scriptsize box}\\
(0.05*0.5) {\large $v$}\\
(0.57*0.5) {\large $v$}\\
(0.45*0.04) $\Omega_{i}$\\
(0.95*0.04) $\Omega_{i}$\\
(0.96*0.5) $\Omega_{i-1}$\\
(0.47*0.5) $\Omega_{i-1}$\\
(0.96*0.73) $x_{n}^{(i-1)}$\\
(0.96*0.6) $x_{1}^{(i-1)}$\\
(0.95*0.3) $x_{n}^{(i)}$\\
(0.95*0.16) $x_{1}^{(i)}$\\
(0.47*0.73) $x_{n}^{(i-1)}$\\
(0.47*0.6) $x_{1}^{(i-1)}$\\
(0.46*0.3) $x_{n}^{(i)}$\\
(0.46*0.16) $x_{1}^{(i)}$\\
%%%%%%%%%%%%%%%%%%%%%%%%%%%%%%%%%%%%%%%%
\endSetLabels
\strut\AffixLabels{\includegraphics*[width=130mm]{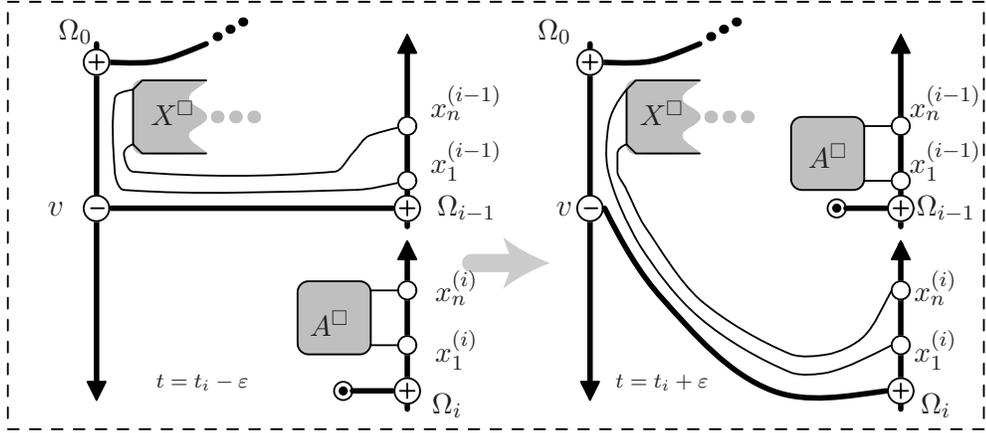}}
\caption{A movie presentation of $D_*$ before and after {\bf Step 3} near $\Omega_{i-1}$ and $\Omega_i$.  } 
\label{fig:step3result}
\end{center}
\end{figure}

{\bf Step 3-1}: Movie for $t\in [t_{i}-\varepsilon, t_{i}+\frac{1}{2}\varepsilon]$ near $\Omega_{i-1}$ and $\Omega_i$.\\

See Figure \ref{fig:step3-1}. 
For each page $S_t$ near $\Omega_{i-1}$ and $\Omega_i$ the slice $D_* \cap S_t$ is exactly the same as $D \cap S_t$ except that a copy of $A$ is added on the right-hand side of $\Omega_i$ and the legs of $X$ are moved to the right-hand side of $ \Omega_{i-1}$.  \\
\begin{figure}[htbp]
\begin{center}
%\ShowGrid
\SetLabels
(0.04*0.92) $\Omega_{0}$\\
(0.56*0.92) $\Omega_{0}$\\
(.3*.95) \scriptsize $t=t_i-\e$\\
(.82*.95) \scriptsize $t=t_i-\e$\\
(0.16*0.85) $X^{\Box}$\\
(0.68*0.85) $X^{\Box}$\\
(0.05*0.88) $x_{1}$\\
(0.05*0.83) $x_{n}$\\
(0.05*0.75) $v$\\
(0.56*0.75) $v$\\
(0.46*0.76) $\Omega_{i-1}$\\
(0.44*0.53) $\Omega_{i}$\\
(0.96*0.53) $\Omega_{i}$\\
(.96*.58) $x^{(i)}_1$\\
(.96*.66) $x^{(i)}_n$\\
(.98*0.85) $x_{n}^{(i-1)}$\\
(.98*0.79) $x_{1}^{(i-1)}$\\
(0.98*0.73) $\Omega_{i-1}$\\
(0.2*0.68) $h_{i}$\\
(0.18*0.65) $+$\\
(0.72*0.68) $h_{i}$\\
(0.68*0.65) $+$\\
(.85*.62) $A^\Box$\\
(.3*.76) $b_t$\\
(.8*.72) $b_t$\\
%%%%%%%%%%%%%%%%%%%%%
(0.04*0.43) $\Omega_{0}$\\
(0.56*0.43) $\Omega_{0}$\\
(.3*.42) \scriptsize $t=t_i+\frac{\e}{2}$\\
(.82*.42) \scriptsize $t=t_i+\frac{\e}{2}$\\
(0.15*0.35) $X^{\Box}$\\
(0.67*0.35) $X^{\Box}$\\
(0.05*0.39) $x_{1}$\\
(0.05*0.32) $x_{n}$\\
(0.05*0.25) $v$\\
(0.56*0.25) $v$\\
(0.46*0.25) $\Omega_{i-1}$\\
(0.44*0.05) $\Omega_{i}$\\
(0.96*0.05) $\Omega_{i}$\\
(.96*.1) $x^{(i)}_1$\\
(.96*.17) $x^{(i)}_n$\\
(0.98*0.36) $x_{n}^{(i-1)}$\\
(0.98*0.3) $x_{1}^{(i-1)}$\\
(0.98*0.25) $\Omega_{i-1}$\\
(0.25*-0.04) (Movie of $D$)\\
(0.75*-0.04) (Movie of $D_*$)\\
(.85*.13) $A^\Box$\\
(.3*.06) $b_t$\\
(.8*.02) $b_t$\\
\endSetLabels
\strut\AffixLabels{\includegraphics*[width=110mm]{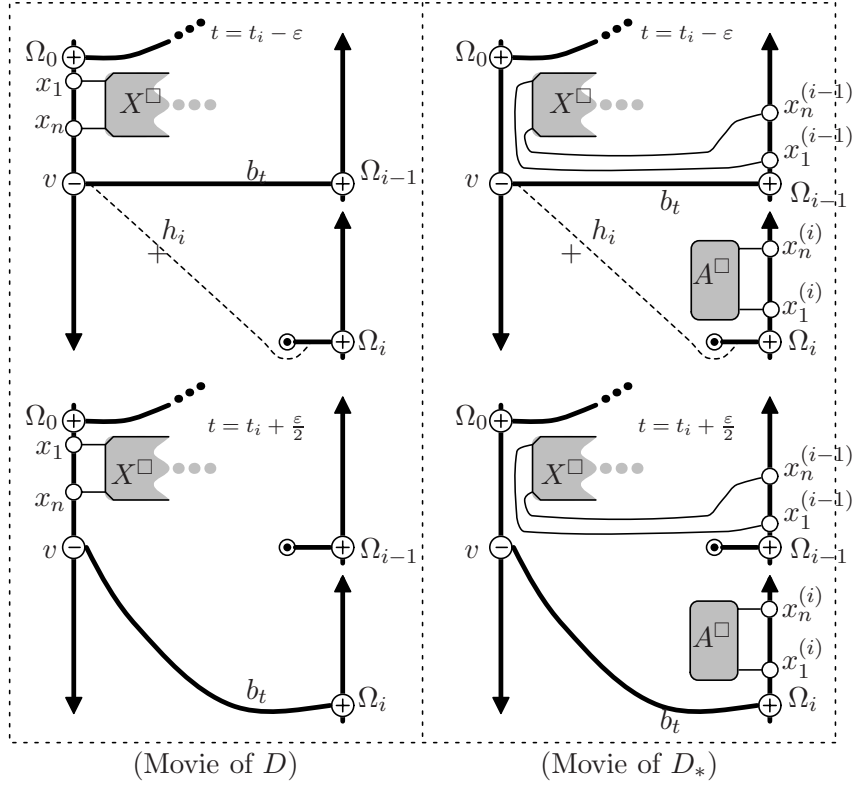}}
\vspace{0.3cm}
\caption{{\bf Step 3-1} for $t\in[t_{i}-\varepsilon, t_{i}+\frac{1}{2}\varepsilon]$.}
\label{fig:step3-1}
\end{center}
\end{figure}

{\bf Step 3-2}: Movie for $t \in (t_i+\frac{\e}{2}, t_i+\e]$ near $\Omega_{i-1}$ and $\Omega_i$. \\

{\bf (3-2-1)}: 
Suppose that the leaf of $\F(D)$ from $x_1$ in the page $S_0$ (i.e., in the set of leaves $A$) is an a-arc. 
Then in the page $S_{t_i+\e/2}$ the leaf from $x_{1}^{(i)}$ must be an a-arc. 
In the page $S_{t_i + \e/2}$
join this a-arc and the leg of $X$ landing on $x_{1}^{(i-1)}$  by a $(+)$ describing arc that lies near $b_{t_i+\e/2}$ as in Figure~\ref{fig:step3-2b}-(1). 
After changing the configuration, the leaf from $x_{1}^{(i-1)}$ becomes an a-arc and $x_1^{(i)}$ is connected to $X$ as in Figure~\ref{fig:step3-2b}-(2). 
\begin{figure}[htbp]
\begin{center}
%\ShowGrid
\SetLabels
(.03*.93) (1)\\
(.54*.93) (2)\\
(0.07*0.85) $\Omega_{0}$\\
(0.56*0.85) $\Omega_{0}$\\
(0.22*0.75) $X^\Box$\\
(0.71*0.75) $X^\Box$\\
(0.08*0.5) $v$\\
(0.56*0.5) $v$\\
(0.47*0.5) $\Omega_{i-1}$\\
(0.45*.08) $\Omega_{i}$\\
(0.95*0.08) $\Omega_{i}$\\
(0.96*0.75) $x_{j}^{(i-1)}$\\
(0.96*0.6) $x_{1}^{(i-1)}$\\
(0.96*0.5) $\Omega_{i-1}$\\
(0.95*0.32) $x_{j}^{(i)}$\\
(0.95*0.16) $x_{1}^{(i)}$\\
(0.47*0.75) $x_{j}^{(i-1)}$\\
(0.47*0.6) $x_{1}^{(i-1)}$\\
(0.47*0.5) $\Omega_{i-1}$\\
(0.46*0.32) $x_{j}^{(i)}$\\
(0.46*0.16) $x_{1}^{(i)}$\\
(0.25*0.42) {\footnotesize $(b_{t_{i}-\varepsilon})$}\\
(0.23*0.05) $b_{t_{i}+\e/2}$\\
%%%%%%%%%%%%%%%%%%%%%%%%%%%%%%%%%%%%%%%%
\endSetLabels
\strut\AffixLabels{\includegraphics*[width=130mm]{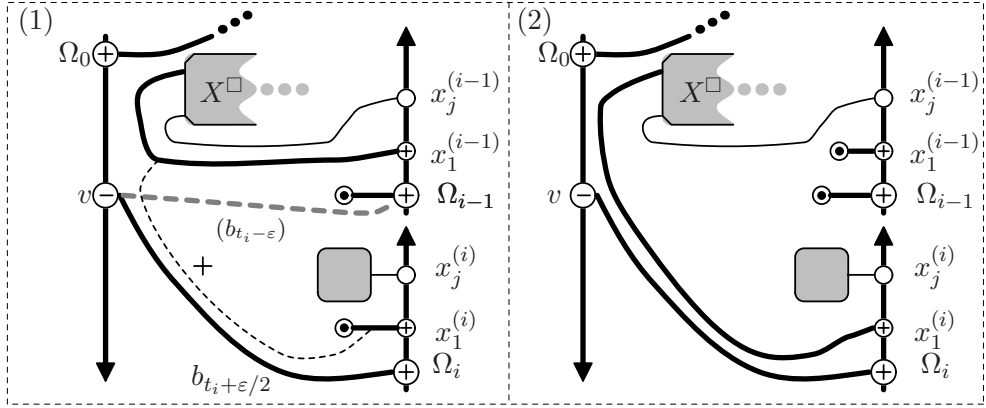}}
\caption{{\bf Step 3-2-1}: Movie of $D_*$.} 
\label{fig:step3-2b}
\end{center}
\end{figure}

{\bf (3-2-2)}: 
Suppose that the leaf of $\F(D)$ from $x_1$ in the page $S_0$ is a b-arc, $b$. 
Let $x_j \in \partial' \Delta$ be the other endpoint of $b$. 
We have $\sgn(x_1) = -\sgn(x_j)$. 
Let $A'$ denote the subset of $A$ enclosed by $b$. 
Put $A'':= A \setminus (A' \cup b)$. 
As a result of {\bf Step 3-1}, in the slice $D_*\cap S_{t_i+\e/2}$ the elliptic points $x_1^{(i)}$ and $x_j^{(i)}$ are joined by a b-arc, $b^{(i)}$, and $b^{(i)}$ encloses a copy of $A'$ as depicted in Figure~\ref{fig:step3-2c}-(1). 
\begin{figure}[htbp]
\begin{center}
%\ShowGrid
\SetLabels
(-.08*1) (1)\\
(.6*1) (2)\\
(.6*.45) (3)\\
(-.08*.45) (4)\\
(0.52*.69) (a)\\
(0.83*0.5) {\large (b)}\\
(0.52*.35) {\large (c)}\\
(-0.02*0.94) $\Omega_{0}$\\
(0.6*0.94) $\Omega_{0}$\\
(0.14*0.9) $X^\Box$\\
(0.75*0.9) $X^\Box$\\
(-0.02*0.8) {$v$}\\
(0.6*0.8) {$v$}\\
(0.42*0.79) $\Omega_{i-1}$\\
(1.04*0.79) $\Omega_{i-1}$\\
(0.4*0.56) $\Omega_{i}$\\
(1.02*0.56) $\Omega_{i}$\\
(0.42*0.9) $x_{j}^{(i-1)}$\\
(0.42*0.83) $x_{1}^{(i-1)}$\\
(1.04*0.9) $x_{j}^{(i-1)}$\\
(1.04*0.83) $x_{1}^{(i-1)}$\\
(0.41*0.68) $x_{j}^{(i)}$\\
(0.41*0.61) $x_{1}^{(i)}$\\
(1.03*0.68) $x_{j}^{(i)}$\\
(1.03*0.61) $x_{1}^{(i)}$\\
(0.14*0.61) {\footnotesize $b_{t_{i}+\e/2}$}\\
(.26*.66) $b^{(i)}$\\
(0.32*0.64) $A'$\\
(0.93*0.64) $A'$\\
(.17*.69) $h$\\
%%%%%%%%%%%%%%%%%%%%%%%%%%%%%%%%%%%%%%%%
(-0.02*0.39) $\Omega_{0}$\\
(0.6*0.39) $\Omega_{0}$\\
(0.14*0.35) $X^\Box$\\
(0.75*0.35) $X^\Box$\\
(-0.02*0.25) $v$\\
(0.6*0.25) $v$\\
(0.42*0.24) $\Omega_{i-1}$\\
(1.04*0.24) $\Omega_{i-1}$\\
(0.4*0.01) $\Omega_{i}$\\
(1.02*0.01) $\Omega_{i}$\\
(0.42*0.35) $x_{j}^{(i-1)}$\\
(0.42*0.27) $x_{1}^{(i-1)}$\\
(1.04*0.35) $x_{j}^{(i-1)}$\\
(1.04*0.27) $x_{1}^{(i-1)}$\\
(0.41*0.12) $x_{j}^{(i)}$\\
(0.41*0.04) $x_{1}^{(i)}$\\
(1.03*0.12) $x_{j}^{(i)}$\\
(1.03*0.04) $x_{1}^{(i)}$\\
(0.32*0.3) $A'$\\
(0.925*0.3) $A'$\\
(.315*.72) $A''$\\
(.315*.16) $A''$\\
(.93*.72) $A''$\\
(.93*.16) $A''$\\
(.23*.28) $b^{(i-1)}$\\
(0.88*0.28) $\overline{h}$ \\
\endSetLabels
\strut\AffixLabels{\includegraphics*[width=120mm]{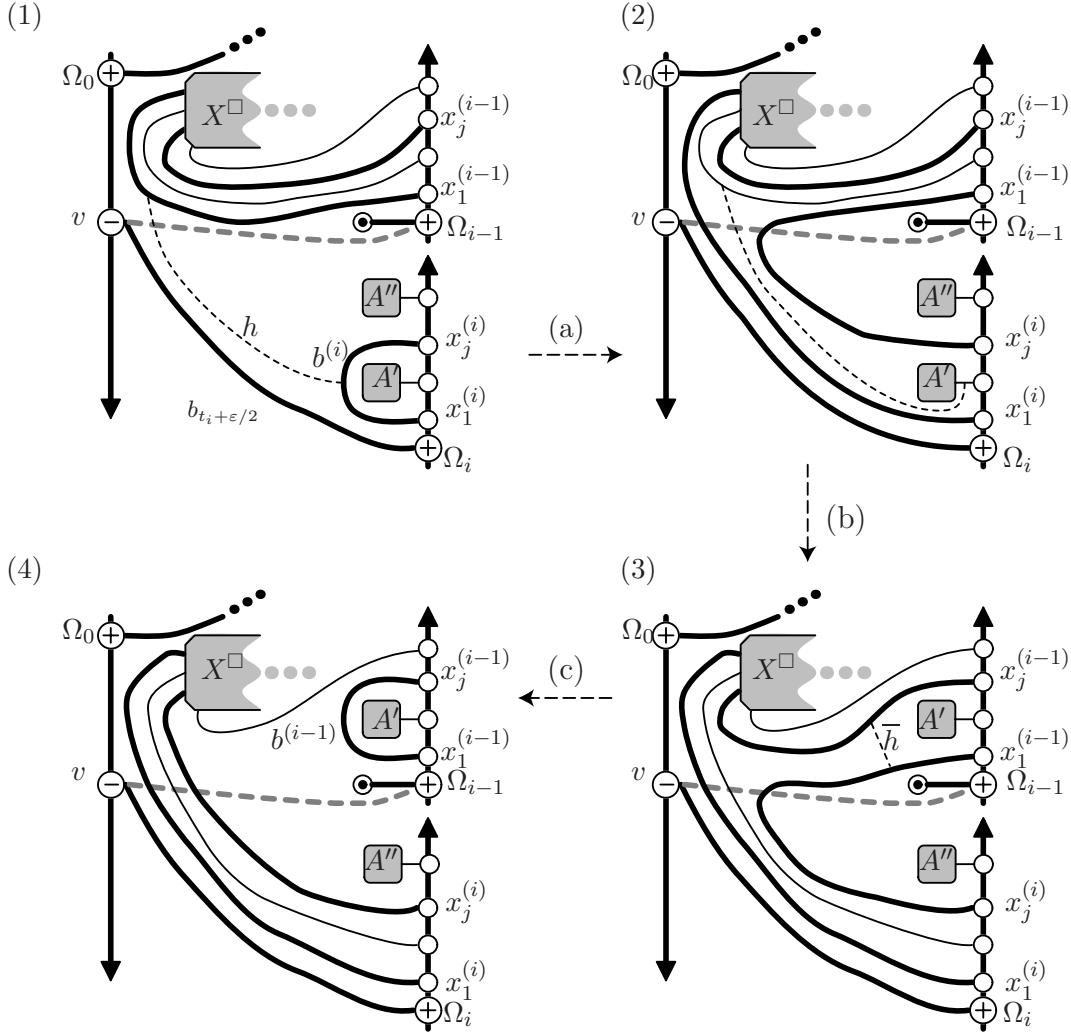}}
\caption{{\bf Step 3-2-2}: Movie of $D_*$.} 
\label{fig:step3-2c}
\end{center}
\end{figure}

\begin{enumerate}
\item[(a)]
Join the b-arc $b^{(i)}$ and the leg $X$ landing on $x_{1}^{(i-1)}$ by a describing arc lying  near $b_{t_i+\e}$ as in Figure~\ref{fig:step3-2c}-(1). 
Call the resulting hyperbolic point $h$. 
Note that $\sgn(h)=\sgn(x_1)$. 
As in the passage (a) of Figure~\ref{fig:step3-2c} the configuration changes so that the leaf box $X$ and $x_1^{(i)}$ are connected, and $x_j^{(i)}$ and $x_1^{(i-1)}$ are joined by a b-arc.
 
\item[(b)] 
Apply to the legs of $A'$ the operations of {\bf Step 3-2-1} or {\bf Step 3-2-2} (a) until $A'$  is completely moved from near $\Omega_i$ to near $\Omega_{i-1}$. 
See the passage (b) of Figure~\ref{fig:step3-2c}.

\item[(c)] 
Introduce a hyperbolic point, $\overline h$, between the leaves from $x_{1}^{(i-1)}$ and $x_{j}^{(i-1)}$ to enclose $A'$ by a b-arc, $b^{(i-1)}$. See the passage (c) of Figure~\ref{fig:step3-2c}. 
The sign of the hyperbolic point $\overline h$ is opposite to that of $h$. Therefore, 
\begin{equation}\label{sign of h}
\sgn(h)=-\sgn(\overline{h}) = \sgn(x_1).
\end{equation}

\item[(d)]
Apply the above (a, b, c) to the leaf box $A''$.
\end{enumerate}

As a consequence, the entire leaf box $A$ is moved to the right-hand side of $\Omega_{i-1}$ from the right-hand side of $\Omega_i$ and the leaves from $x^{(i-1)}_{1},\ldots,x^{(i-1)}_{n}$ are connected to $A$. See Figure~\ref{fig:step3result}. \\

{\bf Step 4}: Movie for  $t\in [1-2\varepsilon,1]$.\\

In the interval $[1-2\varepsilon,1]$ there is only one hyperbolic point $h_{-} \in S_{1-\e}$.  
Its describing arc is depicted in the upper-left sketch of Figure~\ref{fig:step4}. 
Recall {\bf (Step 2)} for $t\in[t_{m}+\e,1-2\varepsilon]$ where  the slice $D_* \cap S_{1-2\varepsilon}$ is obtained from the slice $D \cap S_{1-2\varepsilon}$ by sliding the legs of $A$ along $b_{1-2\varepsilon}$.
We give $D_*$ a negative hyperbolic point (we call it $h_-$ with the same name) in the page $S_{1-\e}$ defined by the describing arc as depicted in the upper-right sketch of Figure~\ref{fig:step4}. 
The region between the describing arc and the binding component $\partial' \Delta$ contains no leaves. 
Therefore, the b-arc $b_0 \subset S_0$ of $D_*$ is boundary parallel in $S_0 \setminus (S_0 \cap D_*)$, in other words $b_0$ is non-essential.  
\begin{figure}[htbp]
\begin{center}
%\ShowGrid
\SetLabels
(.3*.93) \scriptsize $t=1-2\e$\\
(.84*.93) \scriptsize $t=1-2\e$\\
(.3*.4) \scriptsize $t=1$\\
(.84*.4) \scriptsize $t=1$\\
%(0.5*0.72) {\LARGE $S_{1-2\varepsilon}$}\\
(0.035*0.91) $\Omega_{0}$\\
%(0.17*0.78) {\huge \rotatebox{-180}{A}}\\
(.16*.77) \rotatebox{180}{$A^\Box$}\\
(0.04*0.82) $x_{1}$\\
(0.04*0.69) $x_{n}$\\
(0.04*0.6) {\large $v$}\\
(0.43*0.61) $\Omega_{m}$\\
%(0.24*0.65) {\large $-$}\\
(0.24*0.82) $h_{-}$\\
%%%%%%%%%%%%%%%%%%%%%%%%%%%%%%%%%%%%%%%%%%%
%(0.51*0.23) {\LARGE $S_{1}$}\\
(0.03*0.41) $\Omega_{0}$\\
%(0.17*0.29) {\huge \rotatebox{-180}{A}}\\
(.16*.27) \rotatebox{180}{$A^\Box$}\\
(0.04*0.32) $x_{1}$\\
(0.04*0.18) $x_{n}$\\
(0.04*0.06) {$v$}\\
(0.43*0.05) $\Omega_{m}$\\
(.16*.06) $b_0$\\
%%%%%%%%%%%%%%%%%%%%%%%%%%%%%%%%%%%%
(0.57*0.91) $\Omega_{0}$\\
%(0.75*0.79) {\huge \rotatebox{-180}{A}}\\
(.72*.7) $A^\Box$\\
(0.96*0.82) $x_{n}^{(m)}$\\
(0.96*0.69) $x_{1}^{(m)}$\\
(0.56*0.6) {\large $v$}\\
(0.96*0.58) $\Omega_{m}$\\
%(0.68*0.85) {\large $-$}\\
(0.65*0.82) $h_{-}$\\
%%%%%%%%%%%%%%%%%%%%%%%%%%%%%%%%%%%%
(0.57*0.41) $\Omega_{0}$\\
(0.56*0.06) {\large $v$}\\
(0.96*0.05) $\Omega_{m}$\\
(0.96*0.32) $x_{n}^{(m)}$\\
(0.96*0.18) $x_{1}^{(m)}$\\
(.7*.06) $b_0$\\
(0.88*0.2) $A^\Box$\\%{\huge A}\\
(0.25*-0.08) (Movie of $D$)\\
(0.75*-0.08) (Movie of $D_*$)\\
\endSetLabels
\strut\AffixLabels{\includegraphics*[width=130mm]{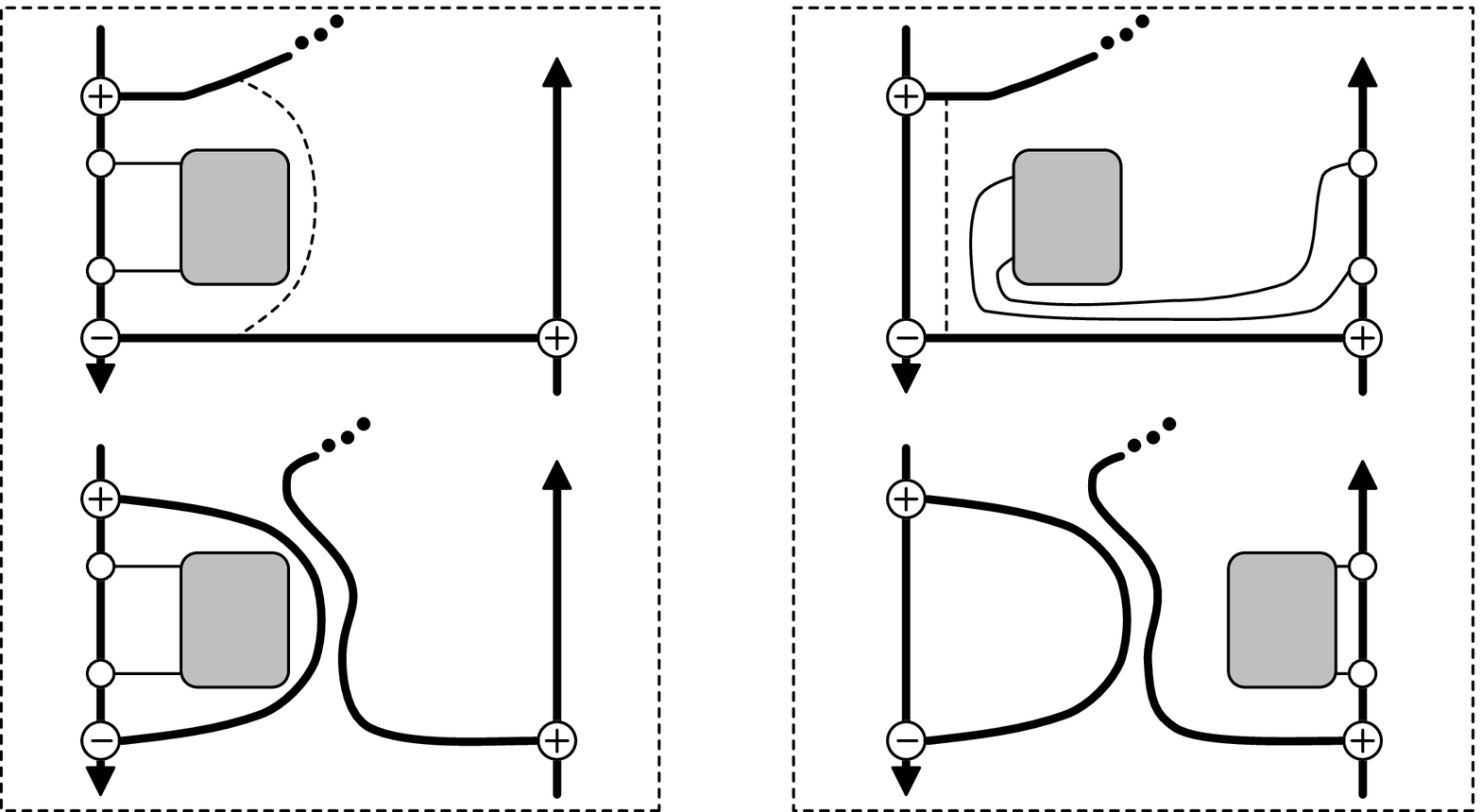}}
\vspace{0.4cm}
\caption{{\bf Step 4:}} 
\label{fig:step4}
\end{center}
\end{figure}

Away from $\partial'\Delta \cup b_{1-2\e}$ we define $D_*$ so that the slices $D\cap S_t$ and $D_* \cap S_t$ are identical except for the added $m$ copies of $A$ on the right-hand side of $\Omega_i$ for each $i=1, \dots, m-1$. See Figure~\ref{fig:step2-2}. 
The copies of $A$ on the page $S_1$ and the copies of $A$ on the page $S_0$ can be identified under the monodoromy $\phi$ since all these copies of $A$ are very close to the binding and $\phi={\rm id}$ near the binding.

Thus $\phi(D_* \cap S_1) = D_* \cap S_0$. We glue $D_* \cap S_1$ and $D_* \cap S_0$ by $\phi$ and obtain a properly embedded surface $D_*$ with non-essential b-arc $b_0$. 
In the next section (Proposition~\ref{prop:foliation}) we will show that $D_*$ is indeed a disc by studying how the open book foliation of $D_*$ is related to the open book foliation of $D$. 
This completes the construction of $D_*$.

\section{Open book foliation of $D_*$}
\label{sec:foliation}

In this section we describe the open book foliation of $D_*$ and explain how $\F(D_*)$ is related to $\F(D)$.
To construct $D_*$ recall that we started with $D$ and removed a neighborhoods of the elliptic points $x_1,\ldots,x_n \in \partial'\Delta$ then added new elliptic points $x^{(i)}_1,\ldots,x^{(i)}_n$ for $i=1,\dots,m$.

\begin{proposition}\label{prop:foliation}
The open book foliation of $D_*$ is obtained from the open book foliation of $D$ with the following changes.
\begin{enumerate}
\item
Near the negative elliptic point $v$ the open book foliation $\F(D_*)$ is identical to $\F(D)$ as depicted in Figure~\ref{fig:nearv}

\item
If the leaf from $x_j$ contained in the page $S_0$ is an a-arc then the change near $x_j$ can be depicted as in Figure~\ref{fig:blowup}-(a). 
In other words, a neighborhood of $x_j$ is replaced by a region consisting of $m$ positive elliptic points and $m-1$ positive hyperbolic points. 

\item
If $x_j$ and $x_k$ are joined by a b-arc in $S_0$ then the change near $x_j$ and $x_k$ can be depicted as in Figure~\ref{fig:blowup}-(b),
In other words, a neighborhood of the b-arc is replaced by a region containing $m$ positive and $m$ negative elliptic points and $m-1$ positive and $m-1$ negative hyperbolic points whose signs are determined by the condition (\ref{sign of h}). 

\end{enumerate}

Therefore, the surface $D_*$ is a disc.
\end{proposition}  

\begin{figure}[htbp]
\begin{center}
%\ShowGrid
\SetLabels
(0.0*0.96) (a)\\
(0.0*0.5) (b)\\
(0*.85) $\F(D)$\\
(0*.4) $\F(D)$\\
(.45*.95) $\F(D_*)$\\
(.43*.46) $\F(D_*)$\\
(0.14*0.9) {\small $x_{j}$}\\
(0.14*0.43) {\small$x_{j}$}\\
(0.16*0.1) {\small$x_{k}$}\\
(0.13*0.62) {\footnotesize $t\!=\!0$}\\
(0.15*0.25) {\footnotesize $t\!=\!0$}\\
(0.81*0.9) $+$\\
(0.51*0.8) {\small$x_{j}^{(m)}$}\\
(0.71*0.8) {\small$x_{j}^{(2)}$}\\
(0.88*0.8) {\small$x_{j}^{(1)}$}\\
(0.8*0.4) $+$\\
(0.8*0.1) $-$\\
(0.51*0.4) {\small$x_{j}^{(m)}$}\\
(0.51*0.2) {\small$x_{k}^{(m)}$}\\
(0.73*0.4) {\small$x_{j}^{(2)}$}\\
(0.73*0.19) {\small$x_{k}^{(2)}$}\\
(0.9*0.3) {\small$x_{j}^{(1)}$}\\
(0.9*0.1) {\small$x_{k}^{(1)}$}\\
\endSetLabels
\strut\AffixLabels{\includegraphics*[width=120mm]{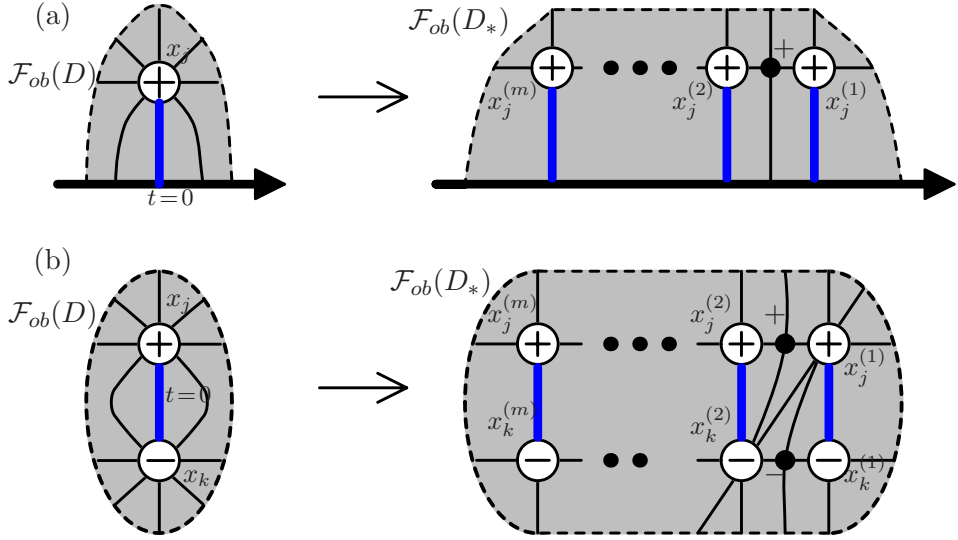}}
\caption{(Proposition~\ref{prop:foliation}-(2), (3)) 
Open book foliations $\F(D)$ and $\F(D_*)$ near $x_j$. 
Here (b) illustrates the case $\sgn(x_{j})=-\sgn(x_{k})=+1$.} 
\label{fig:blowup}
\end{center}
\end{figure}

\begin{proof}[Proof of Proposition~\ref{prop:foliation}-(1)] 
In the construction of $D_*$ we keep the same elliptic points $v, \Omega_0,\dots, \Omega_m$, the b-arcs $b_t$ and the hyperbolic points $h_-, h_1, \dots, h_m$. 
Therefore $\F(D_*)$ and $\F(D)$ are identical near $v$. 
\end{proof}

\begin{proof}[Proof of Proposition~\ref{prop:foliation}-(2)]

Let $j \in \{1,\dots,n\}$. 
Suppose that the leaf from the elliptic point $x_j$ of $D$ contained in the page $S_0$ (hence contained in the leaf box $A^\Box$) is an a-arc. 
This imposes $\sgn(x_j)=+1$, hence $\sgn(x^{(i)}_j)=\sgn(x_j) = +1$ for all $i=1,\dots,m$. 
In the following we compare the changes in foliations and {\bf Step 1}, \dots, {\bf Step 4} of the construction of $D_*$ in Section~\ref{sec:niceposition}.

\noindent{\bf The open book foliation for {\bf Step 1} } 

\noindent Let $t\in [0, t_1+\e]$. 
Recall the operations in Figure~\ref{fig:step1} and Figure~\ref{fig:step1-2}.
For each page $S_t$ we remove the positive elliptic point $x_j$ and the a-arc emanating from $x_j$  then add positive elliptic points $x^{(1)}_j,\dots,x^{(m)}_j$ and a-arcs emanating from $x^{(1)}_j,\dots,x^{(m)}_j$. 
The union of these a-arcs for $t \in [0, t_1+\e]$ is $m$ disjoint union of fan-shaped regions with the pivotal points $x^{(1)}_j,\dots,x^{(m)}_j$ as depicted in the top row of Figure~\ref{fig:step3fola}. 
%%%
\begin{figure}[htbp]
\begin{center}
%\ShowGrid
\SetLabels
(0.02*1.01) {\bf [Step 1]}\\
(0.12*0.96) {$x_j^{(i)}$}\\
(0.03*0.85) {\scriptsize $t=0$}\\
(0.2*0.92) {\scriptsize $t=t_1+\varepsilon$}\\
%(0.4*0.85) {\scriptsize $t=0$}\\
%(0.57*0.92) {\scriptsize $t=t_1+\varepsilon$}\\
(0.02*0.52) {\scriptsize $t=0$}\\
(0.19*0.41) {\scriptsize $t=t_i\!+\!\varepsilon$}\\
(0.2*0.58) {\scriptsize $t=t_{i+1}\!-\!\varepsilon$}\\
(0.04*0.58) {$x_j^{(k)}$}\\
(0.06*0.64) {\scriptsize $(k\neq i)$}\\
(0.02*0.7) {\bf [Step 2]}\\
(0.31*0.48) {\scriptsize $t=0$}\\
(0.48*0.39) {\scriptsize $t=t_i\!+\!\varepsilon$}\\
(0.35*0.7) {\scriptsize $t=t_{i+1}\!-\!\varepsilon$}\\
(0.33*0.55) {$x_j^{(i)}$}\\
(0.82*0.39) {\scriptsize $t=t_i\!+\!\varepsilon$}\\
(0.76*0.7) {\scriptsize $t=t_{i+1}\!-\!\varepsilon$}\\
(0.68*0.51) {$x_j$}\\
(0.9*0.67) {$\mathcal{F}_{ob}(D)$}\\
(0.04*0.22) {$x_j^{(k)}$}\\
(0.06*0.28) {\scriptsize $(k \neq i,i-1)$}\\
(0.02*0.35) {\bf [Step 3]}\\
(0.01*0.16) {\scriptsize $t=0$}\\
(0.17*0.07) {\scriptsize $t=t_i\!-\!\varepsilon$}\\
(0.19*0.23) {\scriptsize $t=t_{i}\!+\!\varepsilon$}\\
(0.31*0.15) {\scriptsize $t=0$}\\
(0.53*0.15) {\scriptsize $t=0$}\\
(0.32*0.23) {$x_j^{(i)}$}\\
(0.54*0.25) {$x_j^{(i-1)}$}\\
(0.49*0.33) {\scriptsize $t=t_i\!-\!\varepsilon$}\\
(0.33*0.28) {\scriptsize $t=t_{i}\!+\!\varepsilon$}\\
(0.81*.23) {\small This space will be filled by}\\
(0.81*.2) {\small a-arcs by Steps 2, 3 and 4.}\\
(.55*.03) {\scriptsize $t=t_i\!+\!\varepsilon$}\\
(0.25*0.03) {\scriptsize $t=t_i\!-\!\varepsilon$}\\
\endSetLabels
\strut\AffixLabels{\includegraphics*[width=130mm]{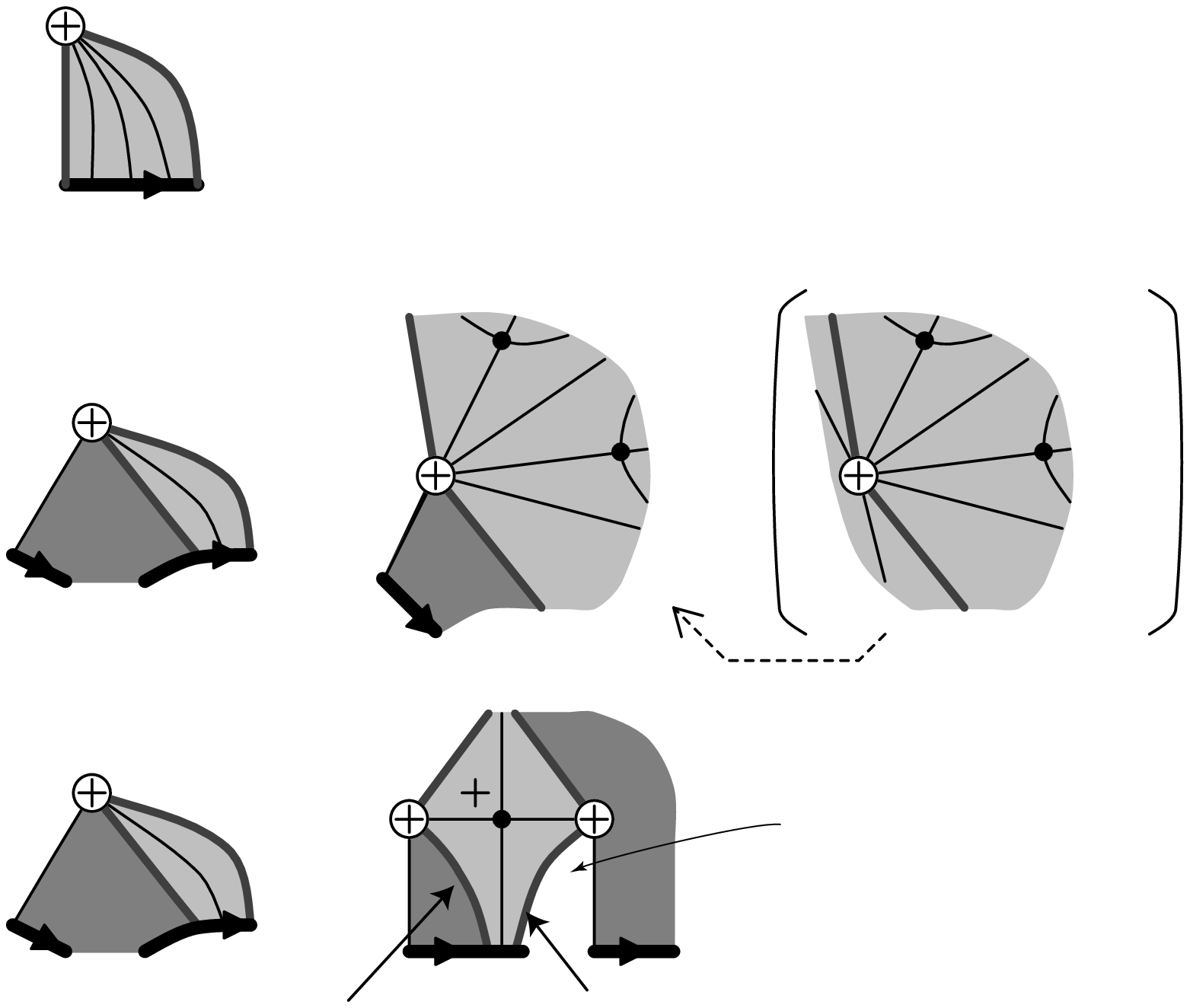}}
\caption{(Proposition~\ref{prop:foliation}-(2)) Open book foliation of $D_*$ near $x_j^{(i)}$, in the case the leaf from $x_j$ in the page $S_0$ is an a-arc.} 
\label{fig:step3fola}
  \end{center}
\end{figure}

\noindent{\bf The open book foliation for {\bf Step 2} } 

\noindent Let $i \in \{1, \dots, m-1\}$ and $t \in [t_i+\e, t_{i+1}-\e)$. 
A similar argument works for $i=m$.

Suppose that $k \neq i$. 
In the page $S_t$ we give $D_*$ an a-arc emanating from $x^{(k)}_j$ as depicted in Figure~\ref{fig:step2-2}. 
The union of the a-arcs for $t \in [t_i+\e, t_{i+1}-\e)$ is a fan-shaped region with the pivotal point $x^{(k)}_j$ (see the light gray region in Figure~\ref{fig:step3fola}).

Suppose that $k = i$. 
The construction depicted in Figures~\ref{fig:slideleg} and \ref{fig:step2} shows that the union of leaves emanating from the elliptic point $x^{(i)}_j$ for the interval $[t_i+\e, t_{i+1}-\e)$ (the light gray region in the middle sketch of Figure~\ref{fig:step3fola}) and the union of leaves emanating from $x_j$ (the right sketch of Figure~\ref{fig:step3fola}) are exactly the same.

\noindent{\bf The open book foliation for  {\bf Step 3} } 

\noindent 
Fix $i \in \{2, \dots, m\}$ and let $t \in [t_i-\e, t_i+\e)$. 

Suppose that $k \neq i-1$, $i$.  
By the operation depicted in Figure~\ref{fig:step2-2} the leaf emanating from $x^{(k)}_j$ in the page $S_t$ is an a-arc.
The union of these a-arcs for the interval $[t_i-\e, t_i+\e)$ forms a fan-shaped region with the pivotal point $x^{(k)}_j$ (the light gray region in the left sketch of  Figure~\ref{fig:step3fola}).

Next we study the open book foliation near $x^{(i-1)}_j$ and $x^{(i)}_j$. 
Let $t\in[t_{i}-\varepsilon, t_{i}+\frac{1}{2}\varepsilon]$.
Due to {\bf Step 3-1} (Figure~\ref{fig:step3-1}) the leaves in $S_0$ from $x^{(i-1)}_j$ and $x^{(i)}_j$ are regular.  
In the interval $[t_{i}+\frac{\e}{2},t_{i}+\e]$, 
{\bf Step 3-2-1} (Figure~\ref{fig:step3-2b}) introduces a positive hyperbolic point in order to switch the foot of the leg of $X$ from $x_j^{(i-1)}$ to $x_j^{(i)}$. 
The light gray region in the bottom right sketch of Figure~\ref{fig:step3fola} is filled.

\noindent{\bf The open book foliation for {\bf Step 4} } 

\noindent
{\bf Step 4} does nothing special to the a-arcs emanating from $x^{(i)}_j$ for every $i=1,\dots, m$. 
Therefore the family of a-arcs for the interval $[1-2\e, 1]$ form a fan-shaped region with the pivotal point $x^{(i)}_j$. 
The a-arc in the page $S_1$ and the a-arc in $S_0$ are identified under the monodromy $\phi$. 
\end{proof}

\begin{proof}[Proof of Proposition~\ref{prop:foliation}-(3)]

Assume that the elliptic points $x_j$ and $x_k$ ($1\leq j < k \leq n$) of $\F(D)$ are joined by a b-arc in the page $S_0$ (hence the b-arc is contained in the leaf box $A^\Box$). 

\noindent{\bf The open book foliation for {\bf Step 1} } 

\noindent
Let $t\in [0, t_1+\e]$. 
In every page $S_t$, {\bf Step 1} adds a copy of $A$ on the right-hand side of $\Omega_i$ for each $i=1,\dots, m$, thus the elliptic points $x^{(i)}_j$ and $x^{(i)}_k$ are joined by a b-arc. 
This yields $m$ disjoint bigons foliated by b-arcs as shown in the top row of Figure \ref{fig:step3b}. 
\begin{figure}[htbp]
 \begin{center}
%\ShowGrid
 \SetLabels
(0.01*1.01) {\bf [Step 1]}\\
(0.14*0.97) {$x_j^{(i)}$}\\
(0.14*0.8) {$x_k^{(i)}$}\\
(0.01*0.94) {\scriptsize $t=0$}\\
(0.2*0.92) {\scriptsize $t=t_1+\varepsilon$}\\
(0.66*0.97) {$x_k^{(i)}$}\\
(0.66*0.8) {$x_j^{(i)}$}\\
(0.54*0.94) {\scriptsize $t=0$}\\
(0.72*0.92) {\scriptsize $t=t_1+\varepsilon$}\\
(0.01*0.75) {\bf [Step 2]}\\
(0.0*0.64) {\scriptsize $t=0$}\\
(0.09*0.6) {\scriptsize $t=t_i\!+\!\varepsilon$}\\
(0.16*0.73) {\scriptsize $t=t_{i+1}\!-\!\varepsilon$}\\
(0.16*0.44) {\scriptsize $t=t_{i+1}\!-\!\varepsilon$}\\
(0.05*0.68) {$x^{(i)}_j$}\\
(0.05*0.47) {$x^{(i)}_k$}\\
(0.32*0.6) {\scriptsize $t=t_i\!+\!\varepsilon$}\\
(0.355*0.72) {\scriptsize $t=t_{i+1}\!-\!\varepsilon$}\\
(0.355*0.44) {\scriptsize $t=t_{i+1}\!-\!\varepsilon$}\\
(0.27*0.68) {$x_j$}\\
(0.27*0.5) {$x_k$}\\
(0.43*0.77) {$\mathcal{F}_{ob}(D)$}\\
(0.01*0.4) {\bf [Step 3]}\\
(0.00*0.25) {\scriptsize $t=0$}\\
(0.07*0.2) {\scriptsize $t=t_i\!-\!\varepsilon$}\\
(0.07*0.36) {\scriptsize $t=t_{i}\!+\!\varepsilon$}\\
(0.28*0.14) {\scriptsize $t=0$}\\
(0.06*0.3) {$x^{(i)}_j$}\\
(0.27*0.29) {$x^{(i-1)}_j$}\\
(0.05*0.07) {$x^{(i)}_k$}\\
(0.27*0.08) {$x^{(i-1)}_k$}\\
(0.24*0.36) {\scriptsize $t=t_i\!-\!\varepsilon$}\\
(0.25*0.18) {\scriptsize $t=t_{i}\!+\!\varepsilon$}\\
(0.07*0.02) {\scriptsize $t=t_{i}\!+\!\varepsilon$}\\
(0.24*0.02) {\scriptsize $t=t_i\!-\!\varepsilon$}\\
(0.42*0.34) {\scriptsize This space}\\
(0.42*0.31) {\scriptsize will be filled}\\
(0.42*0.28) {\scriptsize by b-arcs}\\
(0.42*0.25) {\scriptsize by Steps}\\
(0.42*0.22) {\scriptsize 2, 3 and 4.}\\
(0.52*0.64) {\scriptsize $t=0$}\\
(0.61*0.6) {\scriptsize $t=t_i\!+\!\varepsilon$}\\
(0.68*0.73) {\scriptsize $t=t_{i+1}\!-\!\varepsilon$}\\
(0.68*0.44) {\scriptsize $t=t_{i+1}\!-\!\varepsilon$}\\
(0.57*0.68) {$x^{(i)}_k$}\\
(0.57*0.47) {$x^{(i)}_j$}\\
(0.84*0.6) {\scriptsize $t=t_i\!+\!\varepsilon$}\\
(0.875*0.72) {\scriptsize $t=t_{i+1}\!-\!\varepsilon$}\\
(0.875*0.44) {\scriptsize $t=t_{i+1}\!-\!\varepsilon$}\\
(0.79*0.68) {$x_k$}\\
(0.79*0.5) {$x_j$}\\
(0.95*0.77) {$\mathcal{F}_{ob}(D)$}\\
(0.54*0.26) {\scriptsize $t=0$}\\
(0.59*0.2) {\scriptsize $t=t_i\!-\!\varepsilon$}\\
(0.07*0.36) {\scriptsize $t=t_{i}\!+\!\varepsilon$}\\
(0.8*0.14) {\scriptsize $t=0$}\\
(0.58*0.3) {$x^{(i)}_k$}\\
(0.79*0.29) {$x^{(i-1)}_k$}\\
(0.57*0.07) {$x^{(i)}_j$}\\
(0.79*0.08) {$x^{(i-1)}_j$}\\
(0.76*0.36) {\scriptsize $t=t_i\!-\!\varepsilon$}\\
(0.77*0.18) {\scriptsize $t=t_{i}\!+\!\varepsilon$}\\
(0.59*0.02) {\scriptsize $t=t_{i}\!+\!\varepsilon$}\\
(0.76*0.02) {\scriptsize $t=t_i\!-\!\varepsilon$}\\
(1*0.34) {\scriptsize This space will be}\\
(1*0.31) {\scriptsize filled by b-arcs by}\\
(1*0.28) {\scriptsize Steps 2, 3 and 4.}\\
\endSetLabels
\strut\AffixLabels{\includegraphics*[width=120mm]{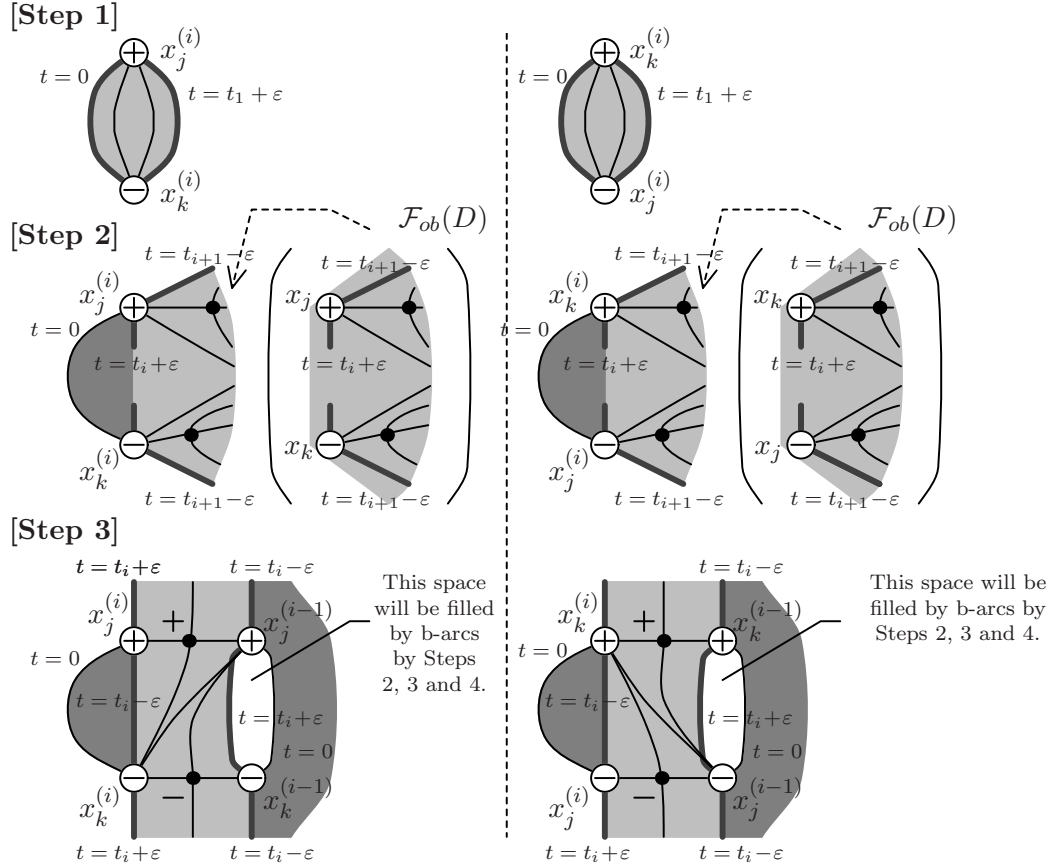}}
\caption{(Proposition~\ref{prop:foliation}-(3)) 
The open book foliation of $D_*$ near $x^{(i)}_{j}$ and $x^{(i)}_{k}$, in the case that $x_j$ and $x_k$ $(j<k)$ are joined by a b-arc in the page $S_0$. The left (resp. right) column illustrates the case $\sgn(x_j)=-\sgn(x_k)=+1$ (resp. $-1$). }
\label{fig:step3b}
  \end{center}
\end{figure}

\noindent{\bf The open book foliation for {\bf Step 2} } 

\noindent
Let $i \in \{1, \dots, m-1\}$ and $t \in [t_i+\e, t_{i+1}-\e)$. 
A similar argument works for $i=m$.   

Suppose that $l\neq i$. 
In the page $S_t$ we give $D_*$ a b-arc joining $x^{(l)}_j$ and $x^{(l)}_j$ as depicted in Figure~\ref{fig:step2-2}. 
The family of the b-arcs for $t \in [t_i+\e, t_{i+1}-\e)$ yields a bigon.

Suppose that $l=  i$. 
The construction depicted in Figures~\ref{fig:slideleg} and \ref{fig:step2} shows that the 
open book foliation of $D$ near $x_j$ and $x_k$ and the open book foliation of $D_*$ near $x^{(l)}_j$ and $x^{(l)}_k$ are identical. 
See the light gray region in the middle row of Figure~\ref{fig:step3b}.

\noindent{\bf The open book foliation for  {\bf Step 3} } 

\noindent 
Let $i \in \{2, \dots, m\}$, and let $l \neq i-1$, $i$ and $t \in [t_i-\e, t_i+\e)$. 
By the operation depicted in Figure~\ref{fig:step2-2} in the page $S_t$ the elliptic points $x^{(l)}_j$ and $x^{(l)}_k$ are joined by a b-arc. 
The family of the b-arcs for the interval $[t_i-\e, t_i+\e)$ yields a bigon.

Let $l = i-1$, $i$.  
In the page $S_t$ for $t\in[t_{i}-\varepsilon, t_{i}+\frac{1}{2}\varepsilon]$ due to {\bf Step 3-1} (Figure~\ref{fig:step3-1}) $x^{(l)}_j$ and $x^{(l)}_k$ are joined by a b-arc. 
The family of the b-arcs for the interval $[t_{i}-\varepsilon, t_{i}+\frac{1}{2}\varepsilon]$ yields a bigon. 
In the interval $[t_{i}+\frac{\e}{2},t_{i}+\e]$, 
{\bf Step 3-2-2} (Figure~\ref{fig:step3-2c}) introduces two hyperbolic points $h$ and $\overline h$ in order to switch the feet of the legs of $X$ from $x_j^{(i-1)}$ to $x_j^{(i)}$ and from $x_k^{(i-1)}$ to $x_k^{(i)}$. 
If $\sgn(x_{j})=-\sgn(x_{k})=+1$, by the sign constraint (\ref{sign of h}) we have $\sgn(h)=-\sgn(\overline h)=+1$. 
Putting {\bf Step 3-1} and {\bf Step 3-2-2} together we obtain a region consisting of two tiles of type ab- or bb-, wee the light gray regions shown in the bottom left sketch of Figure~\ref{fig:step3b}. 
If $\sgn(x_{j})=-\sgn(x_{k})=-1$ a parallel argument holds. 

\noindent{\bf The open book foliation for {\bf Step 4} } 

\noindent
{\bf Step 4} does nothing special to the b-arcs joining $x^{(i)}_j$ and $x^{(i)}_k$ for every $i=1,\dots, m$. 
Therefore the family of the b-arcs for the interval $[1-2\e, 1]$ form a bigon. 
The b-arc in the page $S_1$ and the b-arc in $S_0$ are identified under the monodromy $\phi$. 
\end{proof}

\begin{proposition}
\label{prop:consequence}
The disc $D_*$ has the following properties. 
\begin{enumerate}
\item[(i)] The b-arc $b_0 \subset S_0$ of $\F(D_*)$ connecting $v$ and $\Omega_{0}$ is not essential. 

\item[(ii)] $G_{--}(D_*)$ is a tree and the number of the valence $1$ vertices of $G_{--}(D_*)$ and that of $G_{--}(D)$ are the same.
\item[(iii)] The self-linking number $sl(\partial D_*, D_*)= 1$, 
i.e., $\partial D_*$ does not satisfy the Bennequin-Eliashberg inequality \cite{el2}. 
\end{enumerate}
\end{proposition}

\begin{proof}
(i): By {\bf Step 4} of the construction of $D_*$, the b-arc $b_{0}$ is not essential. 

(ii): 
Suppose that $x_k \in \partial'\Delta$ is a negative elliptic point, that is $x_k$ is a vertex of the graph $G_{--}(D)$. 
As shown in Figure~\ref{fig:blowup}-(b) the graph $G_{--}(D_*)$ is obtained by replacing the vertex $x_k \in G_{--}(D)$ with the linear graph that connects $x_{j}^{(1)},\dots, x_{j}^{(m)}$ and contains $m-1$ negative hyperbolic points. 

(iii):
The self-linking number of a braid $K$ with respect an open book $(S, \phi)$ bounding a Seifert surface $F$ can be computed by using the combinatorial formula proven in \cite[Proposition 3.2]{ik1-1}
$$sl(K, F) = -(e_+ - e_-) + (h_+ - h_-)$$ 
where $e_{\pm}$ is the number of $\pm$ elliptic points of the open book foliation $\F(F)$ and $h_{\pm}$ is the number of $\pm$ hyperbolic points of $\F(F)$. 
By Definition~\ref{def:trans-ot-disc} we have $sl(\partial D, D)=1$. 
By (ii) we conclude $sl(\partial D_*, D_*)=sl(\partial D, D)=1$. 
\end{proof}

\begin{remark-number}\label{rmk non-planar}
If $S$ is a non-planar surface, it may be possible that $\Omega_{i} \in \partial' \Delta$ for some $i$ and exist a non-separating b-arc connecting $v$ and $\Omega_{i}$, i.e., the property {\bf (P2)} may not hold. 
Then the elliptic points $x_{1}^{(i)},\ldots,x_{n}^{(i)}$ of $D_*$ are placed on $\partial' \Delta$ and  Proposition~\ref{prop:consequence}-(i) may not hold.
\end{remark-number}

\section{The new transverse overtwisted disc $D'$ and complexity}
\label{sec:complexity}

In this section we construct a transverse overtwisted disc $D'$. 
Recall the intermediate disc $D_*$ constructed in Section  \ref{sec:niceposition} and studied in Section~ \ref{sec:foliation}. 
By Proposition \ref{prop:consequence}-(i), the b-arc $b_{0} \subset S_0$ is not essential and co-bounds a disc $\Delta$ with $\partial'\Delta$. 
We remove $\Delta$ as shown in Figure~\ref{fig:removeb} and flatten the local extrema. 
As a result the negative elliptic point $v \in \F(D_*)$ is removed. 
Call the resulting disc $D_{**}$. 
The open book foliation changes as in the passage (1)$\to$(2)$\to$(3) of Figure~\ref{fig:near_v}. 
\begin{figure}[htbp]
\begin{center}
%\ShowGrid
\SetLabels
(0.25*0.55) $\Delta$\\ 
(.25*.68) $b_0$\\
(.08*.22) $v$\\
(.05*.72) $\Omega_0$\\
(.25*.15) $D_*$\\
\endSetLabels
\strut\AffixLabels{\includegraphics*[scale=0.5, width=80mm]{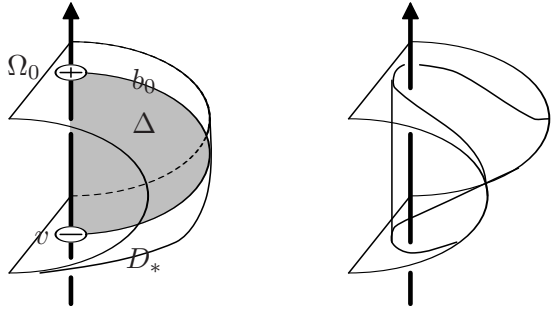}}
\caption{Removing the non-essential b-arc $b_0$.}
\label{fig:removeb}
\end{center}
\end{figure}
\begin{figure}[htbp]
\begin{center}
%\ShowGrid
\SetLabels
(.18*1.02) (1) $\F(D_*)$\\
(0.23*0.2) $\Omega_0$\\
(0.1*0.18) $\Omega_1$\\
(0.06*0.7) $\Omega_{m\!-\!1}$\\
(0.23*0.73) $\Omega_m$\\
(.12*.5) $v$\\
(0.47*1.02) (2)\\
(0.43*0.2) $\Omega_1$\\
(0.41*0.7) $\Omega_{m\!-\!1}$\\
(0.58*0.76) $\Omega_m$\\
(0.85*1.02) (3) $\F(D_{**})$\\
(0.78*0.2) $\Omega_1$\\
(0.76*0.7) $\Omega_{m\!-\!1}$\\
(0.92*0.76) $\Omega_m$\\
\endSetLabels
\strut\AffixLabels{\includegraphics*[width=140mm]{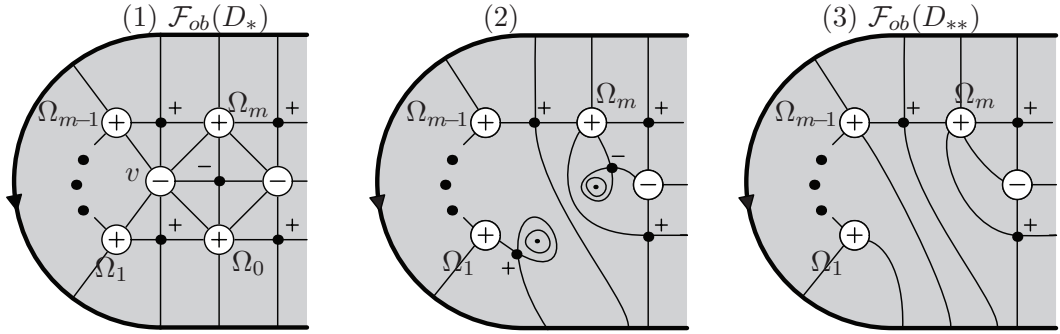}}
\caption{Construction of $D_{**}$. }
\label{fig:near_v}
\end{center}
\end{figure}

This does not affect the self-linking number and we have $sl(\partial D_{**}, D_{**})=sl(\partial D_*,D_*) = 1$ by Proposition~\ref{prop:consequence}-(iii). 
The Bennequin-Eliashberg inequality \cite{el2} does not hold for $D_{**}$. 
Thus, we can apply the construction discussed in \cite[Section 4]{ik1-1} to $D_{**}$ and obtain a transverse overtwisted disc $D'$. 
We call the whole construction $$D\to D_*\to D_{**} \to D'$$ {\em deforming $D$ at $v$}.

\begin{proposition}
\label{prop:ends}
The number of the valence $1$ vertices of the graph $G_{--}(D')$ is less than or equal to that of $G_{--}(D)$. 
\end{proposition}

\begin{proof}
Note that $G_{--}(D') = G_{--}(D_{**})$.
By the construction of $D_{**}$ the graph $G_{--}(D_{**})$ is obtained from $G_{--}(D_*)$ by removing the vertex $v$ and the edge from $v$. 
The assertion follows from Proposition \ref{prop:consequence}-(iii).
\end{proof}

\begin{remark-number}\label{remark on D_*}
Here are remarks on $D_*$ and $D_{**}$.
\begin{enumerate}
\item 
If $m=1$, that is, if there is only one positive hyperbolic point  connected to $v$, then $D_{**}$ is already a transverse overtwisted disc so $D'=D_{**}$. In this case, the operation $D\to D'$ is nothing but an {\em exchange move} studied in \cite{ik3}.

\item  
The passage $D \to D_* \to D_{**}$ does not require that $D$ be a disc. We only need the assumption that $v$ is a non-strongly essential, valence one vertex of $G_{--}(D)$. Similar construction may apply to general surfaces embedded in $M_{(S,\phi)}$.
\end{enumerate}
\end{remark-number}

At a first glance, $\F(D')$ looks more complicated than $\F(D)$, since we have  introduced many singularities, including {\em negative} elliptic points in order to remove the vertex $v \in G_{--}(D)$. 
Our next task is to define a {\em complexity} of a transverse overtwisted disc and show that the complexity of $D'$ is smaller than that of $D$.

\begin{definition}\label{def:B(v)}
Let $D$ be a transverse overtwisted disc. 
Let $$\mathcal V_D := \textrm{the set of valence one, non-strongly essential vertices of } G_{--}(D).$$
For $v \in \mathcal V_D$, the {\em branch} $B(v)$ is the maximal connected subgraph of $G_{--}(D)$ containing $v$ and valence $\leq 2$, non-strongly essential vertices of $G_{--}(D)$.
See Figure \ref{fig:branch}.  
\end{definition}
\begin{figure}[htbp]
 \begin{center}
%\ShowGrid
 \SetLabels
(0.1*0.8) $B(v)$\\
(0.03*0.55) $v$\\ 
(0.85*0.8) $B(u)$\\
(0.98*0.55) $u$\\ 
(.7*1.02) $B(w)$\\
(.75*.88) $w$\\
(0.75*0.4) {\large $G_{--}$}\\
(0.55*0.19) : non-strongly essential negative elliptic point\\
(0.51*0.04) : strongly essential negative elliptic point\\
\endSetLabels
\strut\AffixLabels{\includegraphics*[width=100mm]{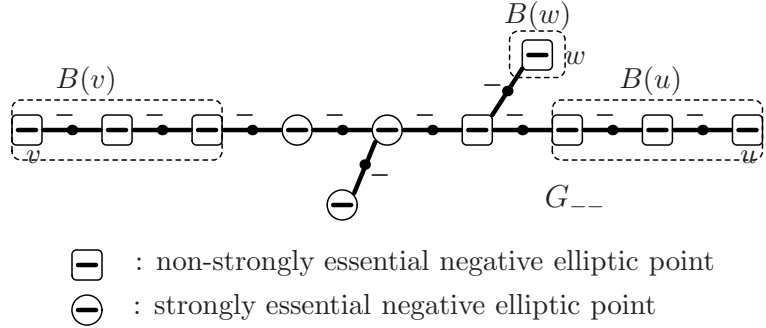}}
\caption{Examples of branches.}
\label{fig:branch}
\end{center}
\end{figure}

\begin{definition}
Let $v \in \mathcal V_D$ and $w \in B(v)$. 
Let $b \subset S_t$ be a non-strongly essential b-arc  ending at $w$. 
That is, $b$ cobounds a disc $\Delta \subset S_{t}$ with a sub-arc $\alpha$, of the binding. We define the {\em nesting level} of $b$ as follows (see also Figure~\ref{fig:nestlevel}): 
\begin{enumerate}
\item If $B(v) \cap \alpha$ is empty then we define the nesting level of $b$ to be zero.
\item If $B(v) \cap \alpha$ is non-empty then let $k$ be the maximal nesting level of the b-arcs in $S_t$ that end at $B(v) \cap \alpha$. The nesting level of $b$ is defined to be $k+1$.
\end{enumerate}
\end{definition}
\begin{figure}[htbp]
 \begin{center}
%\ShowGrid
 \SetLabels
(0.0*0.08) \large $w$\\
(0.4*0.82) \large $\mathsf{2}$\\
(0.18*0.77) \large $\mathsf{0}$\\
(0.165*0.49) \large $\mathsf{0}$\\
(0.21*0.35) \large $\mathsf{0}$\\
(0.38*0.47) \large $\mathsf{1}$\\
(0.75*0.29) \large $\in B(v)$\\
(0.75*0.17) \large $\not \in B(v)$\\
(.5*.5) $b$\\
\endSetLabels
\strut\AffixLabels{\includegraphics*[width=90mm]{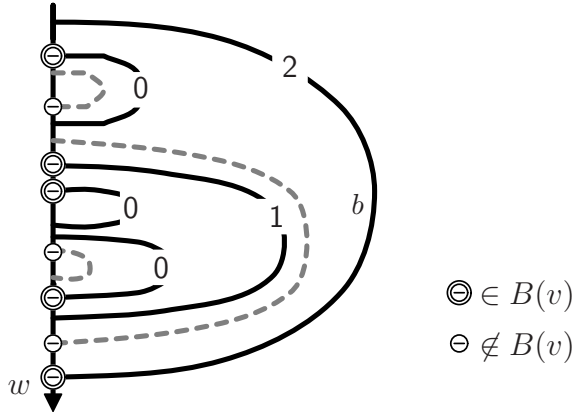}}
\caption{The nesting level of the b-arc $b$ is $2$. 
We take into account only the b-arcs that end at elliptic points in $B(v)$ (bold arcs). B-arcs not ending at $B(v)$  (dashed arcs) do not contribute to the nesting level.} 
\label{fig:nestlevel}
\end{center}
\end{figure}

\begin{definition}\label{NL(w)}
For $v \in \mathcal V_D$ and $w \in B(v)$ we define $\NL_D(w) \in \Z$ the {\em nesting level} of $w$ with respect to $D$ by 
\[ \NL_{D}(w) = \max\{ \textrm{the nesting level of } b \: | \: b  \textrm{ is a non-strongly essential b-arc ending at } w \}. \]
\end{definition}

\begin{definition}
Let $\mC_k$ be the number of the vertices $w$ of $B(v)$ of the nesting level $NL_D(w)=k$.
The {\em complexity} of $v \in \mathcal V_D$ is a sequence of non-negative integers
\[ \mC_{D}(v) = (\ldots,\mC_{k},\mC_{k-1},\ldots,\mC_{1},\mC_0).\]
%By Definition~\ref{NL(w)}, {\color{red} a b-arc of complxity $k$ encloses at least one other b-arc of complexity $k-1$, \marginpar{I feel this reason is not enough to conclude $\mC_k \leq \mC_{k-1}$}} so we have $\mC_{k} \leq  \mC_{k-1}$ for every $k=1, 2, \cdots$. 
We compare sequences by the lexicographical order (from the left). 
For example $$(\ldots,0,0,1,1,4) > (\ldots,0,0,9,42).$$ 
\end{definition}

\begin{definition}
Let 
\[ \mC_D :=  \min \{\mC_{D}(v)\: | \: v \in \mathcal V_D\}. \]
If the set $\mathcal V_D$ is non-empty then $\mathcal C_D > (\dots, 0, 0).$
Thus we may define $\mC_D=(\dots,0,0)$ if and only if all the valence one vertices of $G_{--}(D)$ are strongly essential.
\end{definition}

\begin{definition}
Let $\left| \mathcal V_D \right|$ denote the cardinality of the set $\mathcal V_D$. 
We define the {\em complexity} of $D$ to be the pair
\[ \mathfrak{C}(D):= (\left| \mathcal V_D \right|, \mC_D) \]
and compare $\mathfrak{C}(D)$ by the lexicographical order. 
\end{definition}

\begin{remark-number}
In the definition of $\mC_D$ we use not ``$\max$'' but ``$\min$''. 
The reason becomes apparent in the proof of the following proposition.  To construct the new transverse overtwisted disc $D'$ from $D$ we have added many elliptic points, which may increase the complexity $\mathcal C_D(v)$ of some  $v \in \mathcal{V}_D$ (in particular, $\max \{\mC_{D}(v)\: | \: v \in \mathcal V_D\}$ may increase). 
However, if we look at $v \in \mathcal{V}_D$ realizing the minimal complexity $\mathcal C_D$, deforming at $v$ either reduces $|\mathcal{V}_D|$ or produces a valence one vertex of $G_{--}$ with less complexity.
\end{remark-number}

The following is a key to the proof of Theorem \ref{theorem:main}.

\begin{proposition}
\label{prop:complexity}
Let $D$ be a transverse overtwisted disc with $\left| \mathcal V_D \right| \geq 1$.  
Let $v\in \mathcal V_D$ such that $\mC_D=\mC_{D}(v)$, and $D'$ the transverse overtwisted disc obtained by deforming $D$ at $v$. Then $\mathfrak{C}(D')< \mathfrak{C}(D).$
\end{proposition}

\begin{proof}
By the proof of Proposition~\ref{prop:ends} we have $\left| \mathcal V_{D'} \right| \leq \left| \mathcal V_D \right|$.
If $\left| \mathcal V_{D'} \right| < \left| \mathcal V_D \right|$ then $\mathfrak{C}(D')< \mathfrak{C}(D).$

Assume that $\left| \mathcal V_{D'} \right| = \left| \mathcal V_D \right|$, that is $|B(v)| \geq 2$.  
%Let $v' \in \mathcal V_{D}$ be the vertex which is adjacent to $v \in \mathcal V_D$. 
Let $v' \in B(v)$ be the adjacent vertex to $v \in \mathcal V_D$. % \marginpar{\tiny changed from $v' \in \mathcal V_{D}$ to $v' \in B(v)$} 
By Definition~\ref{def:B(v)} and the construction of $D'$ we have $v' \in \mathcal V_{D'}$. 
To compare $\mathcal C_D$ and $\mathcal C_{D'}$  we examine the nesting levels of the vertices in $B(v')$.
Recall that in the construction of $D_*$ (hence $D'$) we have added copies of the leaf box $A$. 
There are two types of vertices $x\in B(v')$:
\begin{description}
\item[Type A] $x$ is a newly introduced negative elliptic point at a foot of the legs of $A$. (i.e., $x$ is a negative elliptic point of the form $x_{j}^{(i)}$ in Section~\ref{sec:niceposition}.)
\item[Type B] $x$ is a vertex that comes from a  vertex, $x^* \in B(v)$. 
\end{description}

The original leaf box $A$ for $D$ is contained in the disc $\Delta$ co-bounded by the b-arc $b_0$ ending at $v$. So for a vertex $x \in B(v')$ of Type A, we have
$$\NL_{D'}(x) \leq \NL_{D}(v) -1.$$ 
For a vertex $x \in B(v')$ of Type B with the corresponding vertex $x^{*} \in B(v)$, the added copies of $A$ do not affect the nesting level, i.e., 
$$\NL_{D'}(x) = \NL_{D}(x^{*}).$$
To get $D'$ we have removed the vertex $v \in G_{--}(D)$ of the nesting level $\NL_{D}(v)$, so these observations imply: 
$$\mC_{D'} \leq \mC_{D'}(v') < \mC_{D}(v) = \mC_D.$$  
\end{proof}

\section{Proofs of Theorem \ref{theorem:main} \ and Corollary \ref{cor:tight}}
\label{sec:proof}

\begin{proof}[Proof of Theorem \ref{theorem:main}]
Let $(S,\phi)$ be a planar open book supporting an overtwisted contact structure.
Take a transverse overtwisted disc $D \subset M_{(S,\phi)}$.
By \cite[Theorem 3.2]{ik2} we may assume that $\F(D)$ is an essential open book foliation.

If $\F(D)$ contains only one negative elliptic point $v$ we have shown  in \cite[Theorem 6.2, Claim 6.3]{ik2} that $v$ is strongly essential, thus the condition {\bf (SE1)} is satisfied. 

If $\F(D)$ contains more than one negative elliptic point, then by Proposition \ref{prop:complexity} we can find a transverse overtwisted disc $D'$ with $\mathfrak{C}(D')=(0,(\ldots,0))$, i.e., {\bf (SE1)} is satisfied.
\end{proof}

\begin{proof}[Proof of Corollary \ref{cor:tight}]

Suppose that a planar open book $(S,\phi)$ supports an overtwisted contact structure.
By Theorem \ref{theorem:main} there exists a transverse overtwisted disc $D$ with the property {\bf (SE1)}. Let $v \in G_{--}(D)$ be a valence $\leq 1$ vertex and assume that $v$ lies on the binding component $C$. Lemma \ref{lemma:estimate} and {\bf (SE1)} imply that $c(\phi,C) \leq 1$, which is a contradiction. 
\end{proof}

\section{Questions and comments}

As noted in Remark~\ref{remark on D_*}, our construction of the discs $D\to D_*\to D_{**}$ discussed in Sections \ref{sec:niceposition}--\ref{sec:complexity} is valid for general surfaces $F$ in planar open books. 
We call the operation removing a non-strongly essential valence one vertex $v$ {\em deforming $F$ at $v$}. 

An additional argument similar to the one for the exchange move in  \cite{ik3} shows that if $F'$ is a surface obtained by deforming $F$ at $v$, the braids $\partial F$ and $\partial F'$ are {\em transversely isotopic} (if $F$ has boundary) and $F$ and $F'$ are isotopic in $M_{(S,\phi)}$.

This observation, combined with the complexity $\mathfrak{C}(D)$ defined in Section~\ref{sec:complexity}, gives the following result concerning a ``nice'' position of general Seifert surfaces in planar open books. 

\begin{theorem}
Let $F$ be a Seifert surface of a closed braid $L$ with respect to  a planar open book. Then there exists a surface $F'$ isotopic to $F$ such that $\partial F'$ is transversely isotopic to $L$ and 
\begin{description}
\item [(SE1$'$)] all the valence $\leq 1$ vertices of $G_{--}(F')$ are strongly essential.
\end{description}
\end{theorem}

The condition {\bf (SE1$'$)} only concerns the vertices of valence $\leq 1$. One may ask whether one can further modify and put the surface $F$ while preserving the transverse knot type of its boundary so that:
\begin{description}
\item [(SE)] all the vertices of $G_{--}(F)$ are strongly essential.  
\end{description}

We close the paper with a question:
\begin{question}
%Does Theorem \ref{theorem:main} hold for non-planar open books? 
If the FDTC $c(\phi, C) >1$ for all the boundary  components $C \subset \partial S$ then does the open book $(S, \phi)$ support a tight contact structure? 
\end{question}
In our whole arguments, we use the planar assumption only to guarantee the property {\bf (P2)} of Lemma~ \ref{lemma:key}, which is used to deform $D$ at $v$ (cf. Remark~\ref{rmk non-planar}).

\section*{Acknowledgements}
The authors would like to thank Joan Birman, Bill Menasco for useful conversations and  John Etnyre for informing us of an example that is included in Remark \ref{rem:bp}.
They also thank the referee for careful reading and numerous comments and suggestions. 
T.I. was supported by JSPS Research Grant-in-Aid for Research Activity Start-up, Grant Number 25887030.  
K.K. was partially supported by NSF grant  DMS-1206770.

\end{document}